\documentclass[12pt]{article}

\usepackage{color,amsmath,latexsym,amsfonts,amssymb}
\usepackage{booktabs}
\usepackage{amsmath}
\numberwithin{equation}{section}
\usepackage{amsthm}
\usepackage{indentfirst}
\usepackage[utf8]{inputenc}
\usepackage{graphicx}
\usepackage[font=small,labelfont=bf,labelsep=quad]{caption}
\usepackage{float}
\usepackage{subfigure}
\usepackage{multirow}
\usepackage{bbm}
\usepackage{cite}
\usepackage[top=0.8in, bottom=0.8in, left=0.8in, right=0.8in]{geometry}
\usepackage{verbatim}
\usepackage{enumerate}
\usepackage{hyperref}
\usepackage[cmyk]{xcolor}

\def\half{\tfrac{1}{2}}

\def \R{\mathbb{R}}
\def \E{\mathbb{E}}
\def \N{\mathbb{N}}

\def \coeff{\operatorname{coeff}}
\def \d{{\, \rm d}}
\def \P{\mathcal{P}_h}
\def \Cs{\mathcal{C}_{s}}
\def \Cl{\mathcal{C}_{l}}

\newtheorem{theorem}{Theorem}[section]

\newtheorem{lemma}[theorem]{Lemma}
\newtheorem{proposition}[theorem]{Proposition}
\newtheorem{assumption}[theorem]{Assumption}
\newtheorem{example}[theorem]{Example}
\newtheorem{definition}[theorem]{Definition}
\newtheorem{remark}[theorem]{Remark}

\providecommand{\keywords}[1]{\textbf{\text{Keywords: }} #1}
\hypersetup{colorlinks=true,                      
	linkcolor=red,
	anchorcolor=red,
	citecolor=red}

\begin{document}
	
\title{Unconditionally positivity-preserving explicit order-one strong approximations of financial SDEs with non-Lipschitz coefficients
\thanks {This work was supported by Natural Science Foundation of China (12501560, 12401521, 12371417).}}
\author{ Xiaojuan Wu $^\text{a}$, \quad Ruishu Liu $^\text{b,c}${\thanks{Corresponding author: rsliu@szu.edu.cn
}} ,
\quad Jiahao Xu $^\text{d}$
\\
\footnotesize $^\text{a}$ School of Mathematics and Statistics, Hunan First Normal University, Changsha, China\\
\footnotesize $^\text{b}$School of Artificial Intelligence, 
Shenzhen University, China\\
\footnotesize $^\text{c}$ National Engineering Laboratory for Big Data System Computing Technology,
Shenzhen University, China\\
\footnotesize $^\text{d}$ Dundee International Institute, Central South University, Changsha, China
}

\maketitle
	
\begin{abstract}
\footnotesize

In this paper, we are interested in positivity-preserving approximations of stochastic differential equations (SDEs) with non-Lipschitz coefficients, arising from computational finance and possessing positive solutions. 
By leveraging a Lamperti transformation, we develop a novel, explicit, and unconditionally positivity-preserving numerical scheme for the considered financial SDEs.
More precisely, an implicit term $c_{-1}Y_{n+1}^{-1}$ is incorporated in the scheme to guarantee unconditional positivity preservation, and a corrective operator is introduced in the remaining explicit terms to address the challenges posed by non-Lipschitz (possibly singular) coefficients of the transformed SDEs.
By finding a unique positive root of a quadratic equation, the proposed scheme can be explicitly solved and is shown to be strongly convergent with order $1$, when used to numerically solve several well-known financial models such as the CIR process, the Heston-3/2 volatility model, the CEV process and the A\"it-Sahalia model. Numerical experiments validate the theoretical findings.
\end{abstract}
	
	
\bigskip
\normalsize
\section{Introduction}\label{section:introduction}

In the realm of quantitative finance, stochastic differential equations (SDEs) serve as a pivotal tool for modeling various financial phenomena, including asset price dynamics, interest rates, and volatility processes. 
A notable feature of many such models is the presence of non-Lipschitz coefficients, reflecting the complexities and irregularities of real-world financial systems.
As closed-form solutions of these SDEs are rarely available, numerical approximations become necessary in applications.

Although the numerical analysis of SDEs with globally Lipschitz coefficients is well understood (see \cite{kloeden1992numerical,milstein2013stochastic}), numerical approximations of SDEs with non-Lipschitz coefficients meet essential difficulties and are now still an active area.
In 2011, Hutzenthaler, Jentzen, and Kloeden \cite{hutzenthaler2011strong} demonstrated that the widely-used Euler-Maruyama (EM) method (see, e.g., \cite{kloeden1992numerical,HK_book}) produces divergent results when applied to a broad class of SDEs with super-linearly growing coefficients.
Over the past few decades, significant progress has been made in approximating SDEs with super-linearly growing coefficients, by relying on implicit schemes \cite{beyn2016stochastic,beyn2017stochastic,higham2002strong,wang2020mean,andersson2017mean,wang2023mean,szpruch2011numerical,yang2022numerical}, 
or some modified explicit methods based on some strategies such as taming and truncation \cite{hutzenthaler2012strong,sabanis2013note,sabanis2016euler,wang2013tamed,mao2015truncated,tretyakov2013fundamental,mao2016convergence,hutzenthaler2020perturbation,brehier2020approximation,cai2022positivity,yang2024strong,yi2021positivity,jentzen2009pathwise,hutzenthaler2012strong,deng2023positivity,li2024strong,li2024explicit,shi2025convergence,chassagneux2016explicit,liu2025unconditionally,jiang2025unconditionally}.
Albeit computationally efficient, explicit numerical methods often fail to preserve key properties of the exact solution, such as positivity and domain constraints, which are crucial in applications like option pricing and risk management. 
In contrast, some implicit methods have an inherent advantage in preserving positivity and stability (see, e.g., \cite{wang2020mean,wang2023mean,neuenkirch2014first} and references therein), but at the expense of high computational costs due to solving nonlinear systems per step. 
An interesting and natural question thus arises:

\noindent{\bf (Q).} {\it Can one develop positivity-preserving explicit schemes for a class of non-Lipschitz SDEs in computational finance,
with a strong convergence rate revealed?}

In this paper, we attempt to provide a positive answer to this question.
More specifically, we aim to construct a positivity-preserving explicit scheme for scalar SDEs in a general form:
\begin{align}\label{2024explicit-eq:introduction_original_SDE}
\begin{cases}
    \d \mathbb{X}_t
    =
    f(\mathbb{X}_t) \d t
    +
    g(\mathbb{X}_t) \d W_t,
    \quad
    t \geq 0,\\
    \mathbb{X}_0 > 0,
\end{cases}
\end{align}
where $(W_t)_{t\in[0,+\infty)}$ is a standard Brownian motion and coefficients $f,g$ might be non-Lipschitz.
Our approach essentially relies on a Lamperti transformation $\mathbb{L}: (0,+\infty) \rightarrow (0,+\infty)$, which converts the above SDEs with multiplicative noise into transformed ones with additive noise as follows:
\begin{align}\label{2024explicit-eq:introduction_transformed_SDE}
\begin{cases}
    \d X_t
    =
    \mu(X_t) \d t
    +
    \sigma \d W_t, 
    \quad
    t \geq 0,\\
    X_0 
    =
    \mathbb{L}(\mathbb{X}_0).
\end{cases}
\end{align}
In many practical financial models, the transformed drift $\mu$ takes the form of a fractional Laurent polynomial (FLP, see Definition \ref{2024explicit-def:FLP}). 
This special structure provides two key advantages. First, the FLP for the considered financial models satisify the so-called monotonicity condition \eqref{eq:monoton-condition}, which facilitate obtaining a convergence rate for a numerical scheme applied to the transformed SDE \eqref{2024explicit-eq:introduction_transformed_SDE}. Second, the FLP contains a reciprocal function $x^{- \zeta}, \zeta \geq 1$, which helps us design an efficient positivity-preserving with bounded inverse moments (see Lemma \ref{2024explicit-lem:inverse_moment_bounds_Y}).
Thanks to the particular structure of $\mu$, we introduce
\begin{align}\label{2024explicit-eq:introduction_hat_mu_definition}
    \hat \mu(x)
    : =
    \mu(x) 
    -
    c_{-1}  x^{-1},
    \quad
    c_{-1}>0
\end{align}
and split the drift $\mu$ into two parts $\mu(x) = c_{-1}  x^{-1} +  \hat \mu(x)$, which will be numerically treated in a different way. More accurately, we
propose a time-stepping scheme for the transformed SDE \eqref{2024explicit-eq:introduction_transformed_SDE}, on a uniform mesh $\{t_n = n h\}_{n=0}^M, M \in \N$ over $[0,T]$, with a uniform step size $h = \tfrac{T}{M} \in(0,1]$:
\begin{align}
\label{2024explicit-eq:introduction_scheme}
\begin{cases}
    Y_{n+1}
    =
    \P (Y_n) 
    +
    c_{-1} h Y_{n+1}^{-1}
    +
    \hat \mu (\P(Y_n)) h
    +
    \sigma \Delta W_n,
    \:
    n=0,1,...,M-1,\\
    Y_0 = X_0,
\end{cases}
\end{align}
where $\Delta W_n := W_{t_{n+1}} - W_{t_n}$.
The key idea of the scheme \eqref{2024explicit-eq:introduction_scheme} 
lies in treating two parts of the drift coefficient $\mu$ in a different manner. An implicit term $c_{-1}Y_{n+1}^{-1}$ is incorporated for the first part to guarantee the unconditional positivity preserving and a corrective operator $\P$ is introduced in the remaining term $\hat\mu(\P(Y_n))$ to address the challenges posed by non-Lipschitz (possibly singular) coefficients of the transformed SDEs.
This novel design ensures that the scheme benefits from the positivity-preserving property of implicit methods \cite{alfonsi2013strong,neuenkirch2014first,liu2025strong}, while retaining the computational efficiency of explicit methods. 
Indeed, the proposed scheme can be explicitly solved, by finding a unique positive root of a quadratic equation for $Y_{n+1}$.

Based on the numerical approximations of the transformed SDE \eqref{2024explicit-eq:introduction_transformed_SDE}, the numerical approximation of the original SDE \eqref{2024explicit-eq:introduction_original_SDE} is naturally obtained via the inverse transformation $\mathbb{L}^{-1}$:
\begin{equation}
\label{2024explicit-eq:introduction_scheme2}
\mathbb{Y}_n = \mathbb{L}^{-1} (Y_n).
\end{equation}
The error analysis for the approximation $\mathbb{Y}_n$ produced by \eqref{2024explicit-eq:introduction_scheme2} consists of two steps.
As the first step, under general assumptions we provide upper error bounds of the proposed scheme \eqref{2024explicit-eq:introduction_scheme} for the transformed SDE (see Theorem \ref{2024explicit-thm:main_convergence_result}),
which only get involved with the exact solution
processes of the transformed SDE.
Moreover, bounded moments and upper bounds of inverse moments for the numerical solution \eqref{2024explicit-eq:introduction_scheme} are  established in Lemma \ref{2024explicit-lem:Y_integrability} and Lemma \ref{2024explicit-lem:inverse_moment_bounds_Y}, respectively. 
As the second step, we proceed to reveal the strong convergence rate of the proposed scheme for four different financial models: the CIR process, the Heston-3/2 volatility model, the CEV process and the A\"it-Sahalia model. To the end, we carefully analyze the error bounds obtained in the first step for these four models, case by case. Finally, we prove the desired convergence rate of order $1$ for the positivity-preserving scheme applied to these financial models (cf. Propositions \ref{2024explicit-prop:CIR}, \ref{2024explicit-prop:32model}, \ref{2024explicit-prop:CEVmodel} and \ref{2024explicit-prop:ASmodel}).

The rest of this paper is organized as follows.
The next section presents a general setting.
Section \ref{2024explicit-sec:scheme_and_properties} introduces the proposed numerical method and its properties. 
Section \ref{2024explicit-sec:error_analysis} is dedicated to error estimates for the transformed SDE. 
In Section \ref{2024explicit-sec:examples}, we apply our scheme to four different well-known financial models and establish strong convergence results. 
Finally, numerical experiments are reported in Section \ref{2024explicit-sec:numerical_experiments} to verify the theoretical findings.
A short conclusion is provided in Section \ref{2024explicit-section:conclusion}.

\section{Settings}
\label{2024explicit-sec:preliminaries}
Throughout this paper, we let $\mathbb{N}$ denote the set of nonnegative integers and let $T \in (0,+\infty)$, $M \in \mathbb{N}$.
Let $W = \big(W_t\big)_{t \geq 0}$ be a standard Brownian motion defined on a complete filtered probability space $(\Omega, \mathcal{F}, (\mathcal{F}_t)_{t\geq 0}, \mathbb{P})$, where the filtration satisfies usual conditions, and let $ \E $ denote the expectation.
For any two real numbers $a,b$, we denote $a \vee b := \max\{a,b\}$ and $a \wedge b := \min\{a,b\}$.

We begin with the following scalar SDE:
\begin{align}\label{2024explicit-eq:original_SDE}
\begin{cases}
    \d \mathbb{X}_t
    =
    f(\mathbb{X}_t) \d t
    +
    g(\mathbb{X}_t) \d W_t,
    \quad
    t \geq 0,\\
    \mathbb{X}_0 > 0,
\end{cases}
\end{align}
where $f,g: (0,+\infty) \rightarrow \R$ are continuously differentiable functions.
To make the setting general, we impose the following assumption to ensure the well-posedness of this SDE.
\begin{assumption}\label{2024explicit-ass:solution_existence_and_uniqueness}
    The SDE \eqref{2024explicit-eq:original_SDE} has a unique strong solution $\mathbb{X} = \{ \mathbb{X}_t \}_{t \geq 0}$ taking values in $(0,+\infty)$, i.e.,
$
        \mathbb{P}\big(
            \mathbb{X}_t \in (0,+\infty),\, t \geq 0
        \big)
        =
        1.
$
\end{assumption}
If $g(x)>0$, $x \in (0,+\infty)$, we define a Lamperti-type transformation (see \cite{neuenkirch2014first,liu2025strong} for more details) as follows:
\begin{equation}\label{2024explicit-eq:Lamperti_transformation}
	\mathbb{L}(x)
	=
	\sigma \int_{a}^x \tfrac{1}{g(u)} \, d u + b, \quad x \in (0,+\infty),
\end{equation}
where $a\in (0,+\infty)$ and $b \in \R$ are constants selected to simplify the transformation, and $\sigma \neq 0$ is an arbitrary constant.
Applying the transformation $X_t = \mathbb{L} (\mathbb{X}_t) $ to \eqref{2024explicit-eq:original_SDE}, one obtains
\begin{align}\label{2024explicit-eq:transformed_SDE}
    \d X_t
    =
    \mu(X_s) \d s
    +
    \sigma \d W_s,
\end{align}
where 
\begin{equation*}
		\mu(x)
		=
		\sigma \Big(
		\tfrac{f(\mathbb{L}^{-1}(x))}{g(\mathbb{L}^{-1}(x))}
		-
		\half g'(\mathbb{L}^{-1}(x))
		\Big).
\end{equation*}
For many practical financial models, the transformed drift function $\mu$  has a fractional Laurent polynomial (FLP) structure, whose definition and relative notations are given as follows.
\begin{definition}{\cite[Definition 6.1]{liu2025strong}}
\label{2024explicit-def:FLP}
    A rational function $p:(0,+\infty) \rightarrow \mathbb{R}$ is called a fractional Laurent polynomial on $(0, +\infty)$, if there exist
    \begin{enumerate}[(1)]
        \item $e_1,e_2 \in \mathbb{Z}$, {$e_1 \leq e_2$,}
        \item $a_n \in \R$ for $n = e_1,e_1+1, \ldots ,e_2$,
        \item $s_n \in \R$ for $n = e_1,e_1+1, \ldots ,e_2$ with $s_{n_1} \neq s_{n_2}$ for $n_1 \neq n_2$,
    \end{enumerate}
    such that
    \begin{equation*}
        p(x) = \sum_{n = e_1}^{e_2} a_n x^{s_n},
        \quad
        x > 0.
    \end{equation*}
\end{definition}
\begin{definition}{\cite[Definition 6.2]{liu2025strong}}
    For an FLP $p$ on $(0,+\infty)$, set
\begin{align*}
          {\deg}^{+}(p) 
           &  = 
            \max \{ 
                s_n: n = e_1,...,e_2 \, \text{ with } \,  a_n \neq 0 
            \}, \\
           {\coeff}^{+}(p)
            &= 
            a_{n^{+}(p)}
            \end{align*}
            as well as
            \begin{align*}
           {\deg}^{-}(p) 
            & = 
            \min \{ 
                s_n: n = e_1,...,e_2 \, \text{ with } \,  a_n \neq 0 
            \},\\
         {\coeff}^{-}(p)
           & = 
            a_{{n}^{-}(p)}.
    \end{align*}
\end{definition}
Moreover, the transformed drift $\mu$ in many financial models satisfies the following one-sided Lipschitz condition.
\begin{assumption}
\label{2024explicit-ass:mu_one_sided_Lipschitz}
    The continuously differentiable function $\mu: (0, + \infty) \rightarrow \R$ 
    satisfies
    \begin{equation} \label{eq:monoton-condition}
        \big\langle 
            \mu (x) - \mu (y),
            x - y
        \big\rangle
        \leq
        L_0 \vert x - y \vert^2
    \end{equation}
    for some constant $L_0 > 0$.
\end{assumption}
The following lemma specifies when an FLP drift satisfies Assumption \ref{2024explicit-ass:mu_one_sided_Lipschitz}.
\begin{lemma}
\label{2024explicit-lem:FLP_properties}
    Let $g$ be an FLP with ${\deg}^{-}(g) < 0 \leq {\deg}^{+}(g)$.
    Then $g$ is one-sided Lipschitz on $(0,+\infty)$ if and only if one of the following stands:
        \begin{enumerate}[(a)]
            \item $\coeff^{-}(g) > 0$, ${\deg}^{+}(g)\leq 1$.
            \item $\coeff^{-}(g) > 0$, ${\coeff}^{+}(g) < 0$.
        \end{enumerate}    
\end{lemma}
Throughout this work, we will consistently assume $\mu$ to be an FLP and satisfy the one-sided Lipschitz condition.
Within the above framework,
the integrability of the analytical solution to the transformed SDE \eqref{2024explicit-eq:transformed_SDE} can be easily deduced.
Similar results can be found in \cite[Lemma 2.5]{neuenkirch2014first}, the proof of which is omitted here.
\begin{lemma}
\label{2024explicit-lem:X_integrability}
    Let Assumptions \ref{2024explicit-ass:solution_existence_and_uniqueness} and \ref{2024explicit-ass:mu_one_sided_Lipschitz} hold.
    For any $p > 0$, it holds that
    \begin{equation}
        \E \Big[
            \sup_{t \in [0,T]}
            \vert X_t \vert^{p}
        \Big]
        <
        +\infty.
    \end{equation}
\end{lemma}
A priori bounds on the inverse moments of the analytical solution are necessary for the subsequent error analysis. 
To address this, we make the following assumption.
\begin{assumption}
\label{2024explicit-ass:inverse_moment_boundedness_of_analytical_solution}
    For some
    $p^* \in (2,+\infty]$, it holds that
    \begin{equation}
    \begin{split}
        \sup_{t \in [0,T]}
        \E \big[
            \vert X_t \vert^{-p}
        \big]
        & <
        +\infty,
        \quad
        \forall 0 < p \leq p^*.
    \end{split}
    \end{equation}
\end{assumption}

\section{The proposed scheme and its properties}\label{2024explicit-sec:scheme_and_properties}

For $M \in \N$ we construct a uniform mesh $\{t_n = n h\}_{n=0}^M$ over $[0,T]$ with a uniform step size $h = \tfrac{T}{M} \in(0,1]$.
On the uniform mesh, we construct a novel scheme:
\begin{align}
\label{2024explicit-eq:scheme}
    Y_{n+1}
    =
    \P (Y_n) 
    +
    c_{-1} h Y_{n+1}^{-1}
    +
    \hat \mu (\P(Y_n)) h
    +
    \sigma \Delta W_n,
\end{align}
where $\P:\R\rightarrow(0,+\infty)$ is a correction function, $c_{-1}>0$, $\Delta W_n := W_{t_{n+1}} - W_{t_n}$ and
\begin{align}\label{2024explicit-eq:hat_mu_definition}
    \hat \mu(x)
    : =
    \mu(x) 
    -
    c_{-1}  x^{-1}.
\end{align}
The choice of the positive constant $c_{-1}$ depends on the form of the FLP $\mu$, with the aim of ensuring that $\hat\mu$ satisfies the following one-sided Lipschitz condition.
\begin{assumption}
\label{2024explicit-ass:hat_mu_one_sided_Lipschitz}
    For any $x,y > 0$, there exists a positive constant $L$ such that
    \begin{align}\label{2024explicit-eq:hat_mu_one_sided_Lipschitz}
        \big\langle 
            \hat \mu (x) - \hat \mu (y),
            x - y
        \big\rangle
        & \leq
        L \vert x - y \vert^2.
    \end{align}
\end{assumption}
The above assumptions imply the following results.
\begin{lemma}\label{2024explicit-lem:deg_mu}
    Let Assumptions \ref{2024explicit-ass:mu_one_sided_Lipschitz} and \ref{2024explicit-ass:hat_mu_one_sided_Lipschitz} stand.
    Then
    \begin{equation}
        \deg^-(\mu) \leq -1, \quad \coeff^-(\mu) > 0.
    \end{equation}
\end{lemma}
\begin{proof}
    For the first assertion, suppose that $\deg^-(\mu) > -1$.
    Observing that $\hat\mu$ is also an FLP, it can be derived by definition that
    \begin{align*}
        \deg^-(\hat\mu) &= -1,
        \quad 
        \coeff^-(\hat\mu) = - c_{-1},\\
        \deg^-(\hat\mu') &= -2,
        \quad 
        \coeff^-(\hat\mu') =  c_{-1}.
    \end{align*}
    Thus
    $$\lim_{x \rightarrow 0^+}\hat\mu'(x) = + \infty,$$
    violating \eqref{2024explicit-eq:hat_mu_one_sided_Lipschitz}.
    The second assertion follows directly from the first assertion and Assumption \ref{2024explicit-ass:mu_one_sided_Lipschitz}.
\end{proof}
\begin{remark}\label{2024explicit-remark_c-1}
We would like to illustrate how one can choose the value of $c_{-1}$ to guarantee Assumption \ref{2024explicit-ass:hat_mu_one_sided_Lipschitz}.
    \begin{itemize}
        \item in the case $\deg^-(\mu) < -1$, we have $\deg^-(\hat\mu) = \deg^-(\mu) < -1$ and $\coeff^-(\hat\mu) = \coeff^-(\mu)$, which implies that $\hat\mu$ satisfies the same one-sided Lipschitz condition as $\mu$ for any constant $c_{-1}$;
        
        \item in the case $\deg^-(\mu) = -1$, set $c_0 = \coeff^-(\mu)$ and
        $$\mu_0(x) := \mu(x) - c_0 x^{-1}.$$
        If $\mu_0$ satisfies the one-sided Lipschitz condition, then $c_{-1} = c_0 = \coeff^-(\mu)$ is the optimal choice. 
        Otherwise, one may choose any $c_{-1} < c_0 = \coeff^-(\mu)$.
        This ensures $\deg^-(\hat\mu) = \deg^-(\mu) = -1$ and $\coeff^-(\hat\mu) = \coeff^-(\mu) - c_{-1} > 0$, which implies that $\hat\mu$ satisfies the same one-sided Lipschitz condition as $\mu$.
    \end{itemize}
\end{remark}
The correction function $\P$ is designed to address the challenges posed by non-Lipschitz (possibly singular) drift $\mu$, satisfying the following assumptions.
\begin{assumption}
\label{2024explicit-ass:Ph_properties}
For the mapping $\P \colon (0,+\infty)\rightarrow (0,+\infty)$, there exists a positive constant $C$ independent of the time step size $h$ such that the following conditions hold:
\begin{enumerate}
    
    \item (contractivity) it holds that
    \begin{align}
    \label{2024explicit-eq:Ph_contractivity}
        \big\vert \P(x) - \P(y) \big\vert
        & \leq
        \vert x - y \vert,
        \quad
        \forall
        x,y >0.
    \end{align}

    \item (consistency‌) there exists some constants $m_1,m_2 \geq 0$ such that
    \begin{align}
    \label{2024explicit-eq:correction_error}
        \big\vert \P(x) - x \big\vert
        & \leq
        C h^2 (1 + \vert x \vert^{-m_1} + \vert x \vert ^{m_2}),
        \quad
        \forall x > 0.
    \end{align}
\end{enumerate}

\end{assumption}

\begin{assumption}
\label{2024explicit-ass:hat_mu(Ph)_bound}
    For any $x >0$, there exists a positive constant $C$ independent of the time step size $h$ such that 
    \begin{equation}
        \big\vert \hat\mu(\P(x)) \big\vert \leq C h^{-\frac12} \vee C \vert x \vert,
        \quad
        \big\vert \hat\mu(\P(x)) - \hat\mu(\P(y)) \big\vert
        \leq
        C h^{-\frac12} \vert x - y \vert.
    \end{equation}
\end{assumption}

We now analyze fundamental characteristics of the proposed numerical scheme, starting with its preservation of positivity.
\begin{lemma}
\label{2024explicit-lem:positivity_preserving}
Given any initial value $Y_0 = X_0 > 0$ and positive constant $c_{-1}>0$, 
the scheme \eqref{2024explicit-eq:scheme} admits unique, positive numerical solutions 
$\{ Y_n \}_{ n \in \N} $ for any step-size $h = \tfrac{T}{N} \in (0,1]$. 
\end{lemma}
\begin{proof}
It suffices to show that for any constant $c \in \mathbb{R}$, the equation
$$G(x) := x - c_{-1} h x^{-1} = c$$
admits a unique solution in $(0,+\infty)$.
This follows from the following observations:
\begin{itemize}
    \item 
    $\lim\limits_{x \rightarrow 0+} G(x) = - \infty,
    \quad
    \lim\limits_{x \rightarrow +\infty} G(x) = + \infty$;
    \item 
    $G'(x) = 1 + c_{-1} h x^{-2} > 0$, indicating that $G$ is monotonically increasing on $(0, +\infty)$.
\end{itemize}
Hence, a unique solution exists, and the proof is completed.
\end{proof}
The Lamperti inverse transformation inherently requires a moment bound for the numerical solution.
The bound is established in the following lemma.
\begin{lemma}\label{2024explicit-lem:Y_integrability}
    Let Assumptions \ref{2024explicit-ass:hat_mu_one_sided_Lipschitz}, \ref{2024explicit-ass:Ph_properties} and \ref{2024explicit-ass:hat_mu(Ph)_bound} stand.
    For any $p>0$, it holds that
    \begin{equation}
        \E\bigg[ \sup_{n=0,...,M} \vert Y_n \vert^p \bigg] < +\infty.
    \end{equation}
\end{lemma}
\begin{proof}
    To begin with, adding $- Y_0 + c_{-1} h Y_0^{-1}$ to both sides of \eqref{2024explicit-eq:scheme} 
    yields 
    \begin{equation}
    \begin{split}
        Y_{n+1} - Y_0
        -
        c_{-1}h \big( Y_{n+1}^{-1} - Y_0^{-1} \big)
        & =
        \P(Y_n) - \P(Y_0)
        +
        \P(Y_0) - Y_0
        +
        c_{-1}h Y_0^{-1}\\
        & \quad +
        \hat\mu(\P(Y_n)) h 
        +
        \sigma \Delta W_n.
    \end{split}
    \end{equation}
    Squaring both sides leads to
    \begin{equation}\label{2024explicit-eq:Y_integrability_squaring_both_sides}
    \begin{split}
        & \Big\vert Y_{n+1} - Y_0 - c_{-1}h \big( Y_{n+1}^{-1} - Y_0^{-1} \big) \Big\vert^2\\
        & =
        \big\vert \P(Y_n) - \P(Y_0) \big\vert^2
        +
        \big\vert \P(Y_0) - Y_0 + c_{-1} h Y_0^{-1} \big\vert^2
        +
        h^2 \big\vert \hat\mu(\P(Y_n)) \big\vert^2  
        +
        \sigma^2 \vert \Delta W_n \vert^2\\
        & \quad
        +
        2 \big\langle \P(Y_n) - \P(Y_0) , \P(Y_0) - Y_0 + c_{-1} h Y_0^{-1} \big\rangle
        +
        2 h \big\langle \P(Y_n) - \P(Y_0) , \hat\mu(\P(Y_n))\big\rangle\\
        & \quad
        +
        2 \big\langle \P(Y_n) - \P(Y_0) , \sigma \Delta W_n \big\rangle
        +
        2 h \big\langle \P(Y_0) - Y_0 + c_{-1} h Y_0^{-1} , \hat\mu(\P(Y_n)) \big\rangle\\
        & \quad
        +
        2 \big\langle \P(Y_0) - Y_0 + c_{-1} h Y_0^{-1} , \sigma \Delta W_n \big\rangle
        +
        2 h \big\langle \hat\mu(\P(Y_n)), \sigma \Delta W_n \big\rangle
    \end{split}
    \end{equation}
    Lemma \ref{2024explicit-lem:positivity_preserving} along with straightforward calculations gives
    \begin{equation*}
    \begin{split}
        & \Big\vert Y_{n+1} - Y_0 - c_{-1}h \big( Y_{n+1}^{-1} - Y_0^{-1} \big) \Big\vert^2\\
        & \qquad  =
        \Big\vert Y_{n+1} - Y_0 \Big\vert^2
        +
        2 c_{-1}h\frac{\big\vert Y_{n+1} - Y_0 \big\vert^2}{Y_0 Y_{n+1}}
        +
        c_{-1}^2 h^2 \big\vert Y_{n+1}^{-1} - Y_0^{-1} \big\vert^2\\
        & \qquad  \geq
        \big\vert Y_{n+1} - Y_0 \big\vert^2.
    \end{split}
    \end{equation*}
    For the six cross terms in \eqref{2024explicit-eq:Y_integrability_squaring_both_sides}:
    \begin{enumerate}[(1)]
        \item By the Young inequality,
        \begin{equation*}
        \begin{split}
        2 & \big\langle \P(Y_n) - \P(Y_0) , \P(Y_0) - Y_0 + c_{-1} h Y_0^{-1} \big\rangle\\
        & \qquad \leq 
        h \big\vert \P(Y_n) - \P(Y_0) \big\vert^2
        +
        \tfrac{1}{h} \big\vert \P(Y_0) - Y_0 + c_{-1} h Y_0^{-1} \big\vert^2.
        \end{split}
        \end{equation*}
        Note that $Y_0 = X_0 > 0$.
        Inequality \eqref{2024explicit-eq:correction_error} infers that for some $m_1, m_2>0$,
        \begin{equation}\label{2024explicit-eq:PhY-Y_bound}
        \begin{split}
        \big\vert \P(Y_0) - Y_0 \big\vert
        & =
        \big\vert \P(X_0) - X_0 \big\vert
        \leq 
        C h^2
        \big( 1 + \vert X_0 \vert^{-m_1} +  \vert X_0 \vert^{m_2}\big)
        \leq
        C h^2.
        \end{split}
        \end{equation}
        Hence, 
        \begin{equation*}
        \begin{split}
        2 & \big\langle \P(Y_n) - \P(Y_0) , \P(Y_0) - Y_0 + c_{-1} h Y_0^{-1} \big\rangle
        \leq 
        h \big\vert \P(Y_n) - \P(Y_0) \big\vert^2
        +
        C h.
        \end{split}
        \end{equation*}
        \item Assumption \ref{2024explicit-ass:hat_mu_one_sided_Lipschitz} together with the Young inequality infers that
        \begin{equation*}
    \begin{split}
        2 & h \big\langle \P(Y_n) - \P(Y_0) , \hat\mu(\P(Y_n))\big\rangle\\
        & =
        2 h \big\langle \P(Y_n) - \P(Y_0) , \hat\mu(\P(Y_n)) - \hat\mu(\P(Y_0))\big\rangle\\
        & \quad+
        2 h \big\langle \P(Y_n) - \P(Y_0) ,  \hat\mu(\P(Y_0)) \big\rangle\\
        & \leq
        2 L h \big\vert \P(Y_n) - \P(Y_0) \big\vert^2
        +
        h \big\vert \P(Y_n) - \P(Y_0) \big\vert^2
        +
        h \big\vert \hat\mu(\P(Y_0)) \big\vert^2\\
        & \leq
        (2 L h + h) \big\vert \P(Y_n) - \P(Y_0) \big\vert^2
        +
        C h.
    \end{split}
    \end{equation*}

        \item The third term is left unchanged.

        \item By the Young inequality, Assumption \ref{2024explicit-ass:hat_mu(Ph)_bound} and \eqref{2024explicit-eq:PhY-Y_bound}, one gets
        \begin{equation*}
        \begin{split}
        2 & h \big\langle \P(Y_0) - Y_0 + c_{-1} h Y_0^{-1} , \hat\mu(\P(Y_n)) \big\rangle\\
        & \qquad \leq
        \big\vert\P(Y_0) - Y_0 + c_{-1} h Y_0^{-1} \big\vert^2
        +
        h^2 \big\vert \hat\mu(\P(Y_n)) \big\vert^2\\
        & \qquad \leq
        C h.
        \end{split}
        \end{equation*}
        
        \item The Young inequality together with \eqref{2024explicit-eq:PhY-Y_bound} yields
        \begin{equation*}
        \begin{split} 
        2 &\big\langle \P(Y_0) - Y_0 + c_{-1} h Y_0^{-1} , \sigma \Delta W_n \big\rangle\\
        & \qquad \leq
        \big\vert\P(Y_0) - Y_0 + c_{-1} h Y_0^{-1} \big\vert^2
        +
        \sigma^2 \vert \Delta W_n \vert^2\\
        & \qquad \leq C h +\sigma^2 \vert \Delta W_n \vert^2 .
        \end{split}
        \end{equation*}

        \item By the Young inequality, Assumption \ref{2024explicit-ass:hat_mu(Ph)_bound} and \eqref{2024explicit-eq:Wn_moment_bound},
        \begin{equation*}
        \begin{split}
        2 h \big\langle \hat\mu(\P(Y_n)), \sigma \Delta W_n \big\rangle
        \leq
        h^2 \big\vert \hat\mu(\P(Y_n)) \big\vert^2
        +
        \sigma^2 \vert \Delta W_n \vert^2
        \leq
        C h + \sigma^2 \vert \Delta W_n \vert^2.
        \end{split}
        \end{equation*}
    \end{enumerate}
    Substituting the above results into \eqref{2024explicit-eq:Y_integrability_squaring_both_sides} and applying \eqref{2024explicit-eq:Ph_contractivity} leads to
    \begin{equation}\label{2024explicit-eq:en_cross_terms_estimated}
    \begin{split}
        \big\vert & Y_{n+1} - Y_0 \big\vert^2\\
        & \leq
        (1 + 2 h + 2 L h)\big\vert \P(Y_n) - \P(Y_0) \big\vert^2
        +
        2 \big\langle \P(Y_n) - \P(Y_0) , \sigma \Delta W_n \big\rangle\\
        & \quad
        +
        3 \sigma^2 \vert \Delta W_n \vert^2
        +
        C h\\
        & \leq
        (1 + 2 h + 2 L h)\big\vert Y_n - Y_0 \big\vert^2
        +
        2 \big\langle \P(Y_n) - \P(Y_0) , \sigma \Delta W_n \big\rangle
        +
        3 \sigma^2 \vert \Delta W_n \vert^2
        +
        C h.
    \end{split}
    \end{equation}
    For simplicity, denote
    $$S_n : = Y_n - Y_0.$$
    Inequality \eqref{2024explicit-eq:en_cross_terms_estimated} now reads as
    \begin{equation*}
        \vert S_{n+1} \vert^2
        \leq
        (1 + 2 h + 2 L h) \vert S_n \vert^2
        +
        2 \big\langle \P(Y_n) - \P(Y_0) , \sigma \Delta W_n \big\rangle
        +
        3 \sigma^2 \vert \Delta W_n \vert^2
        +
        C h.
    \end{equation*}
    By iteration, one obtains
    \begin{equation}\label{2024explicit-eq:Y_integrability_Sn_iteration}
    \begin{split}
        \vert S_{n+1} \vert^2
        & \leq
        \sum_{i=0}^n 
        2 (1 + 2 h + 2 L h)^{n-i} 
        \big\langle \P(Y_i) - \P(Y_0) , \sigma \Delta W_i \big\rangle\\
        & \quad +
        \sum_{i=0}^n 
        (1 + 2 h + 2 L h)^{n-i} 3 \sigma^2 \vert \Delta W_i \vert^2
        +
        \sum_{i=0}^n 
        (1 + 2 h + 2 L h)^{n-i} 
        C h.
    \end{split}
    \end{equation}
    Note that
        \begin{equation}\label{2024explicit-eq:Wn_moment_bound}
        \E\big[\vert \Delta W_n \vert^{2q}\big] = h^q,
        \quad
        \forall q \geq 1.
        \end{equation}
    Since $\P(Y_n) - \P(Y_0)$ and $\Delta W_n$ are independent random variables for any $n=0,...,M-1$, it holds that
    \begin{equation}\label{2024explicit-eq:Y_integrability_PYn_DeltaWn_independent}
    \begin{split}
        \E \Big[ \big\langle \P(Y_n) - \P(Y_0) , \sigma \Delta W_n \big\rangle \Big]
        & = 0.
    \end{split}
    \end{equation}
    Therefore, taking expectation and supremum on both sides of \eqref{2024explicit-eq:Y_integrability_Sn_iteration} results in 
    \begin{equation}\label{2024explicit-eq:Y_integrability_Sn^2_bound}
    \begin{aligned}
        \sup_{n=0,...,M}
        \E \Big[ \big\vert S_n \big\vert^2\Big] 
        \leq
        \sum_{i=0}^{M-1} (1 + 2 h + 2 L h)^{M-1-i} C h
        \leq
        \text{e}^{(2+2L)T}CT
        < 
        +\infty.
    \end{aligned}
    \end{equation}
    Now take supremum and raise both sides of \eqref{2024explicit-eq:Y_integrability_Sn_iteration} to the power of $2q \, (q \geq 1)$.
    The Jensen inequality yields
    \begin{equation}
    \begin{split}
        \E \Big[
            \sup_{n=0,...,M}
          & \vert S_{n} \vert^{4q}
        \Big]
        \leq
        2^{2q-1} 
        \E\Bigg[ 
            \sup_{n=0,...,M-1} 
            \bigg\vert 
            \sum_{i=0}^n 
            (1 + 2 h + 2 L h)^{n-i} 
            3 \sigma^2 \vert \Delta W_i \vert^2
            \bigg\vert^{2q} 
        \Bigg]\\
        & + 
        2^{2q-1} 
        \E\Bigg[ 
            \sup_{n=0,...,M-1} 
            \bigg\vert 
            \sum_{i=0}^n 
            (1 + 2 h + 2 L h)^{n-i} 
            C h
            \bigg\vert^{2q} 
        \Bigg]
        \\
        & +
        2^{2q-1} 
        \E \Bigg[
            \sup_{n=0,...,M-1} 
            \bigg\vert 
            \sum_{i=0}^n 
            2 (1 + 2 h + 2 L h)^{n-i} 
            \big\langle \P(Y_i) - \P(Y_0) , \sigma \Delta W_i \big\rangle
            \bigg\vert^{2q} 
        \Bigg].
    \end{split}
    \end{equation}
    Note that $\big\{\sum_{i=0}^n 
        2 (1 + 2 h + 2 L h)^{-i} 
        \big\langle \P(Y_i) - \P(Y_0) , \sigma \Delta W_i \big\rangle \big\}_{n=0}^{M-1}$
    is a square-integrable discrete-time martingale.
    Applying the Jensen inequality again on the first term and the Burkholder-Davis-Gundy inequality (Theorem \ref{2024explicit-thm:discrete_BDG}) on the second term leads to
    \begin{equation}\label{2024explicit-eq:supSn^4q_BDG}
    \begin{split}
        \E \Big[ 
            \sup_{n=0,...,M} 
          &  \vert S_{n} \vert^{4q}
        \Big]
        \leq
        2^{2q-1} M^{2q-1} 
        \text{e}^{2 (2 + 2 L)T q}
        \sum_{i=0}^{M-1}
        \big\vert 
        C h
        \big\vert^{2q}\\
        &
        +
        2^{4q-1} 
        \text{e}^{2 (2 + 2 L)T q}
        \E\Bigg[
        \bigg(
        \sum_{i=0}^{M-1}
        \big\vert
        (1 + 2 h + 2 L h)^{-i} 
        \big\langle \P(Y_i) - \P(Y_0) , \sigma \Delta W_i \big\rangle
        \big\vert^{2}
        \bigg)^q
        \Bigg].
    \end{split}
    \end{equation}
    Another Jensen inequality applied to the second term of \eqref{2024explicit-eq:supSn^4q_BDG} results in 
    \begin{equation}\label{2024explicit-eq:supSn^4q_BDGandJensen}
    \begin{split}
        \E \Big[
            \sup_{n=0,...,M}
            \vert S_{n} \vert^{4q}
        \Big]
        & \leq
        2^{2q-1} T^{2q} 
        \text{e}^{2(2 + 2 L)T q}
        C^{2q}\\
        & \quad
        +
        2^{4q-1} M^{q-1}
        \text{e}^{2 (2 + 2 L)T q}
        \sum_{i=0}^{M-1} 
        \E\Big[
        \big\vert
        \big\langle \P(Y_i) - \P(Y_0) , \sigma \Delta W_i \big\rangle
        \big\vert^{2q}
        \Big].
    \end{split}
    \end{equation}
    Recalling that $\P(Y_i) - \P(Y_0)$ and $\Delta W_i$ are independent random variables for any $i=0,...,M-1$, one can deduce from \eqref{2024explicit-eq:Ph_contractivity} and \eqref{2024explicit-eq:Wn_moment_bound} that
    \begin{equation}
    \label{2024explicit-eq:Y_integrability_Sn_iteration_PhY,Delta_W_n}
    \begin{split}
        \E \Big[ \big\vert 
        \big\langle \P(Y_i) - \P(Y_0) , \sigma \Delta W_i \big\rangle 
        \big\vert^{2q}\Big]
        & =
        \sigma^{2q} h^q \E \Big[ \big\vert \P(Y_i) - \P(Y_0) \big\vert^{2q}\Big]\\
        & \leq
        \sigma^{2q} h^q \E \Big[ \big\vert S_i \big\vert^{2q}\Big],
        \quad
        \forall q \geq 1.
    \end{split}
    \end{equation}
    Substituting \eqref{2024explicit-eq:Y_integrability_Sn_iteration_PhY,Delta_W_n} in to \eqref{2024explicit-eq:supSn^4q_BDGandJensen} leads to
    \begin{equation}
    \begin{split}
        \E \Big[
            \sup_{n=0,...,M}
            \vert S_{n} \vert^{4q}
        \Big]
        & \leq
        2^{2q-1} T^{2q} 
        \text{e}^{2(2 + 2 L)T q}
        C^{2q}\\
        & \quad +
        2^{4q-1} M^{q-1} \text{e}^{2 (2 + 2 L)T q}
        \sum_{i=0}^{M-1}
        \sigma^{2q} h^{q}
        \E \Big[\vert S_i \vert^{2q}\Big]\\
        & \leq
        C \bigg(
        1 + h \sum_{i=0}^{M-1}
        \E \Big[\vert S_i \vert^{2q}\Big]
        \bigg).
    \end{split}
    \end{equation}
    Note that
    $$h \sum\limits_{i=0}^{M-1}
        \E \Big[\vert S_i \vert^{2q}\Big] 
        \leq
        T
        \cdot
        \sup\limits_{n=0,...,M} \E \Big[ \big\vert S_n \big\vert^{2q}\Big].$$ 
    Bearing \eqref{2024explicit-eq:Y_integrability_Sn^2_bound} in mind, it can be deduced by setting $q=1$ that
    \begin{equation}
        \E \Big[
            \sup_{n=0,...,M}
            \vert S_{n} \vert^{4}
        \Big]
        \leq
        C \bigg(
        1 + \sup_{n=0,...,M} \E \Big[ \big\vert S_n \big\vert^{2}\Big]
        \bigg)
        \leq C.
    \end{equation}
    Moreover, since
    $$h \sum_{i=0}^{M-1}
        \E \Big[\vert S_i \vert^{2q}\Big]
        \leq
        T \cdot \E \Big[\sup_{n=0,...,M}\big\vert S_n \big\vert^{2q}\Big],$$
    it holds for $q>1$, $q\in \N$ that
    \begin{equation}
    \begin{split}
        \E \Big[
            \sup_{n=0,...,M}
            \vert S_{n} \vert^{4q}
        \Big]
        & \leq
        C \bigg(
        1 +  \E \Big[ \sup_{n=0,...,M} \big\vert S_n \big\vert^{2q}\Big]
        \bigg).
    \end{split}
    \end{equation}
    The proof is thus completed by an induction argument along with the Lyapunov inequality.

\end{proof}

\section{Error bounds for discretizations of the transformed SDEs}
\label{2024explicit-sec:error_analysis}

In this section, we attempt to establish the convergence result for the proposed numerical scheme.
To do so, rewrite the transformed SDE \eqref{2024explicit-eq:transformed_SDE} in a discrete form as
\begin{equation}
\begin{aligned}
\label{2024explicit-eq:SDE_rewrite}
    X_{t_{n+1}}
    =
    \P (X_{t_n}) 
    +
    c_{-1} h X_{t_{n+1}}^{-1}
    +
    \hat \mu (\P(X_{t_n})) h
    +
    \sigma \Delta W_n
    +
    R_{n+1},
\end{aligned}
\end{equation}
where
\begin{equation}\label{2024explicit-eq:Rn_definition}
\begin{aligned}
    R_{n+1}
    =
    X_{t_n} - \P (X_{t_n})
    +
    \int_{t_n}^{t_{n+1}}
        \mu(X_s) 
    \d s
    -
    c_{-1} h X_{t_{n+1}}^{-1}
    -
    \hat \mu(\P(X_{t_n})) h.
\end{aligned}
\end{equation}
In the subsequent analysis, we decompose $R_{n+1}$ into two distinct components:
\begin{equation*}
    R_{n+1} = R_{n+1}^{(1)} + R_{n+1}^{(2)}.
\end{equation*}
The component $R_{n+1}^{(2)}$ includes terms involving stochastic integrals and satisfies
$$
\E \Big[  R_{n+1}^{(2)} \big\vert \mathcal{F}_{t_n}\Big]
= 0,
\quad
\forall 
n = 0,1,..., M-1,
$$
allowing it to be addressed effectively with the Burkholder-Davis-Gundy inequality.
The remaining terms are grouped into $R_{n+1}^{(1)}$, which can be handled directly using simpler tools such as the Jensen inequality, the Young inequality and the H\"older inequality.
The following lemma provides rough moment bounds for $R_{n+1}^{(1)}$ and $R_{n+1}^{(2)}$.
\begin{lemma}\label{2024explicit-lem:Rn_bounds}
Let Assumption \ref{2024explicit-ass:solution_existence_and_uniqueness},
\ref{2024explicit-ass:mu_one_sided_Lipschitz},
\ref{2024explicit-ass:inverse_moment_boundedness_of_analytical_solution},
\ref{2024explicit-ass:hat_mu_one_sided_Lipschitz},
\ref{2024explicit-ass:Ph_properties}
and
\ref{2024explicit-ass:hat_mu(Ph)_bound}
stand.
For any $0< p \leq - \tfrac{p^*}{2 \deg^-(\mu)} \wedge \tfrac{p^*}{2m_1}$, where $p^*>0$ comes from Assumption \ref{2024explicit-ass:inverse_moment_boundedness_of_analytical_solution} and $m_1$ from Assumption \ref{2024explicit-ass:Ph_properties}, it holds that
\begin{equation}
    \sup_{n=1,...,M}
    \E\Big[ \big\vert R_{n}^{(1)} \big\vert^{2p}\Big]
    < +\infty,
    \quad
    \sup_{n=1,...,M}
    \E\Big[ \big\vert R_{n}^{(2)} \big\vert^{2p}\Big]
    < +\infty.
\end{equation}
\end{lemma}
\begin{proof}
Thanks to Lemma \ref{2024explicit-lem:deg_mu} and Assumption \ref{2024explicit-ass:hat_mu(Ph)_bound}, it suffices to show
\begin{equation}\label{2024explicit-eq:mu_x^2p_bound}
    \E\Big[ \big\vert X_t \big\vert^{-2m_1p} \Big]
    +
    \E\Big[ \big\vert \mu(X_t) \big\vert^{2p} \Big] < +\infty.
\end{equation}
By Definition \ref{2024explicit-def:FLP} and Lemma \ref{2024explicit-lem:X_integrability}, inequality \eqref{2024explicit-eq:mu_x^2p_bound} reduces to
$$
    \E\Big[ \big\vert X_t \big\vert^{-2m_1p} \Big]
    +
    \E\Big[ \big\vert X_t \big\vert^{2p \deg^-(\mu)} \Big] < +\infty.
$$
An application of Assumption \ref{2024explicit-ass:inverse_moment_boundedness_of_analytical_solution} finishes the proof.
\end{proof}

Now we present the crucial error estimate.
\begin{theorem}\label{2024explicit-thm:main_convergence_result}
Let Assumptions 
\ref{2024explicit-ass:solution_existence_and_uniqueness}, 
\ref{2024explicit-ass:mu_one_sided_Lipschitz}, 
\ref{2024explicit-ass:inverse_moment_boundedness_of_analytical_solution},
\ref{2024explicit-ass:hat_mu_one_sided_Lipschitz},
\ref{2024explicit-ass:Ph_properties}
and
\ref{2024explicit-ass:hat_mu(Ph)_bound}
stand.
For any $1 \leq  p \leq - \tfrac{p^*}{2 \deg^-(\mu)} \wedge \tfrac{p^*}{2m_1}$, where $p^*>0$ comes from Assumption \ref{2024explicit-ass:inverse_moment_boundedness_of_analytical_solution} and $m_1$ from Assumption \ref{2024explicit-ass:Ph_properties}, it holds that
    \begin{equation}
    \begin{split}
        \E\Big[ \sup_{n=0,...,M} \big\vert X_{t_{n}} - Y_n \big\vert^{2p} \Big]
        & \leq
        C
        \Bigg(
            \tfrac{1}{h^{2p}}
            \sup_{n = 1,...,M}
            \E \Big[ \big\vert R_{n}^{(1)} \big\vert^{2p} \Big]
            +
            \tfrac{1}{h^{p}}
            \sup_{n = 1,...,M}
            \E \bigg[
            \big\vert  R_{n}^{(2)} \big\vert^{2p}
            \bigg] 
        \Bigg).
    \end{split}
    \end{equation}
\end{theorem}

\begin{proof}
Combining \eqref{2024explicit-eq:scheme} and \eqref{2024explicit-eq:SDE_rewrite} yields
\begin{equation}
\begin{aligned}
\label{2024explicit-eq:X-Y_beginning}
    & X_{t_{n+1}} - c_{-1}h X_{t_{n+1}}^{-1} - Y_{n+1} + c_{-1}h Y_{n+1}^{-1}\\
    & \qquad \qquad
    =
    \P(X_{t_n}) - \P(Y_n)
    +
    \big( \hat\mu(\P(X_{t_n})) - \hat\mu(\P(Y_n)) \big) h
    +
    R_{n+1}
\end{aligned}
\end{equation}
for any $n = 0,...,M-1$. 
Squaring both sides of \eqref{2024explicit-eq:X-Y_beginning} leads to
\begin{equation}\label{2024explicit-eq:Xn-Yn_square}
\begin{aligned}
    \big\vert X_{t_{n+1}} & - c_{-1}h X_{t_{n+1}}^{-1} - Y_{n+1} + c_{-1}h Y_{n+1}^{-1} \big\vert^2\\
    & =
    \big\vert 
        \P(X_{t_n}) - \P(Y_n)
        +
        \big( \hat\mu(\P(X_{t_n})) - \hat\mu(\P(Y_n)) \big) h
        +
        R_{n+1}
    \big\vert^2\\
    & =
    \big\vert \P(X_{t_n}) - \P(Y_n) \big\vert^2
    +
    \big\vert  \hat\mu(\P(X_{t_n})) - \hat\mu(\P(Y_n))  \big\vert^2 h^2
    +
    \big\vert R_{n+1} \big\vert^2\\
    & \quad
    +
    2 h \big\langle \P(X_{t_n}) - \P(Y_n), \hat\mu(\P(X_{t_n})) - \hat\mu(\P(Y_n)) \big\rangle\\
    & \quad
    +
    2 \big\langle \P(X_{t_n}) - \P(Y_n), R_{n+1} \big\rangle\\
    & \quad
    +
    2 h \big\langle \hat\mu(\P(X_{t_n})) - \hat\mu(\P(Y_n)), R_{n+1} \big\rangle.
\end{aligned}
\end{equation}
Note that
\begin{equation*}
\begin{aligned}
    \big\vert X_{t_{n+1}} - c_{-1}h X_{t_{n+1}}^{-1} - Y_{n+1} + c_{-1}h Y_{n+1}^{-1} \big\vert^2
    & =
    \big\vert X_{t_{n+1}} - Y_{n+1} \big\vert^2
    +
    2 c_{-1} h \tfrac{\big\vert X_{t_{n+1}} - Y_{n+1} \big\vert^2}{X_{t_{n+1}}Y_{n+1}}\\
    & \quad +
    c_{-1}^2 h^2
    \big\vert X_{t_{n+1}}^{-1} - Y_{n+1}^{-1} \big\vert^2\\
    & \geq
    \big\vert X_{t_{n+1}} - Y_{n+1} \big\vert^2.
\end{aligned}
\end{equation*}
Denote
$e_n : = X_{t_n} - Y_n$ for brevity. 
Then \eqref{2024explicit-eq:Xn-Yn_square} reduces to
\begin{equation}\label{2024explicit-eq:en_brevity}
\begin{aligned}
    \big\vert e_{n+1} \big\vert^2
    & \leq
    \big\vert \P(X_{t_n}) - \P(Y_n) \big\vert^2
    +
    \big\vert  \hat\mu(\P(X_{t_n})) - \hat\mu(\P(Y_n))  \big\vert^2 h^2
    +
    \big\vert R_{n+1} \big\vert^2\\
    & \quad
    +
    2 h \big\langle \P(X_{t_n}) - \P(Y_n), \hat\mu(\P(X_{t_n})) - \hat\mu(\P(Y_n)) \big\rangle\\
    & \quad
    +
    2 \big\langle \P(X_{t_n}) - \P(Y_n), R_{n+1} \big\rangle\\
    & \quad
    +
    2 h \big\langle \hat\mu(\P(X_{t_n})) - \hat\mu(\P(Y_n)), R_{n+1} \big\rangle.
\end{aligned}
\end{equation}
Assumption \ref{2024explicit-ass:hat_mu_one_sided_Lipschitz} infers that
\begin{equation*}
    2 h \big\langle \P(X_{t_n}) - \P(Y_n), \hat\mu(\P(X_{t_n})) - \hat\mu(\P(Y_n)) \big\rangle
    \leq
    2 L h \big\vert \P(X_{t_n}) - \P(Y_n) \big\vert^2.
\end{equation*}
By the Young inequality, it holds that
\begin{equation*}
\begin{split}
    2 \big\langle \P(X_{t_n}) - \P(Y_n), R_{n+1}^{(1)} \big\rangle
    & \leq
    h \big\vert \P(X_{t_n}) - \P(Y_n) \big\vert^2
    +
    \tfrac{1}{h} \big\vert R_{n+1}^{(1)} \big\vert^2,\\
    2 h \big\langle \hat\mu(\P(X_{t_n})) - \hat\mu(\P(Y_n)), R_{n+1} \big\rangle
    & \leq
    \big\vert \hat\mu(\P(X_{t_n})) - \hat\mu(\P(Y_n)) \big\vert^2 h^2
    +
    \big\vert R_{n+1} \big\vert^2.
\end{split}
\end{equation*}
Combining the above estimations, one obtains from \eqref{2024explicit-eq:en_brevity} that
\begin{equation}
\begin{aligned}
\label{2024explicit-eq:en_after_Young_estimate}
    \big\vert e_{n+1} \big\vert^2
    & \leq
    (1 + 2 L h + h) \big\vert \P(X_{t_n}) - \P(Y_n) \big\vert^2
    +
    2 \big\vert  \hat\mu(\P(X_{t_n})) - \hat\mu(\P(Y_n))  \big\vert^2 h^2\\
    & \quad +
    2 \big\vert R_{n+1} \big\vert^2
    +
    \tfrac{1}{h} \big\vert R_{n+1}^{(1)} \big\vert^2
    +
    2 \big\langle \P(X_{t_n}) - \P(Y_n), R_{n+1}^{(2)} \big\rangle.
\end{aligned}
\end{equation}
For the second term in \eqref{2024explicit-eq:en_after_Young_estimate}, Assumption \ref{2024explicit-ass:hat_mu(Ph)_bound} implies that
\begin{equation}\label{2024explicit-eq:hat_mu_Ph_Xn-Yn_estimate}
    \big\vert  \hat\mu(\P(X_{t_n})) - \hat\mu(\P(Y_n))  \big\vert^2 h^2
    \leq
    C h \vert e_n \vert^2.
\end{equation}
Bearing \eqref{2024explicit-eq:Ph_contractivity} in mind, substituting \eqref{2024explicit-eq:hat_mu_Ph_Xn-Yn_estimate} into \eqref{2024explicit-eq:en_after_Young_estimate} leads to
\begin{equation}
\begin{aligned}
    \big\vert e_{n+1} \big\vert^2
    & \leq
    (1 + 2 L h + h + 2 C h) \big\vert e_n \big\vert^2
    +
    2 \big\vert R_{n+1} \big\vert^2
    +
    \tfrac{1}{h} \big\vert R_{n+1}^{(1)} \big\vert^2 \\
    & \quad 
    +
    2 \big\langle \P(X_{t_n}) - \P(Y_n), R_{n+1}^{(2)} \big\rangle,
\end{aligned}
\end{equation}
By iteration, one obtains
\begin{equation}\label{2024explicit-eq:en_iretation}
\begin{aligned}
    \big\vert e_{n+1} \big\vert^2
    & \leq
    \sum_{i=0}^{n}
    (1 + 2 L h)^{n-i}
    (2 C+1) h \vert e_i \vert^2
    +
    \sum_{i=0}^{n}
    (1 + 2 L h)^{n-i}
    2 \big\vert R_{i+1} \big\vert^2\\
    & \quad 
    +
    \sum_{i=0}^{n}
    (1 + 2 L h)^{n-i}
    \tfrac{1}{h} \big\vert R_{i+1}^{(1)} \big\vert^2 
    +
    \sum_{i=0}^{n}
    (1 + 2 L h)^{n-i}
    2 \big\langle \P(X_{t_i}) - \P(Y_i), R_{i+1}^{(2)} \big\rangle.
\end{aligned}
\end{equation}
Here the term $(2C+1)h \vert e_n \vert^2$ is retained without iteration in order to facilitate the application of a type of Gronwall's inequality (Lemma \ref{2024explicit-lem:Gronwall}), which will be shown in the forthcoming derivation.
For $p \geq 1$ and $l \in \{0,1,...,M-1\}$, raising both sides of \eqref{2024explicit-eq:en_iretation} to the power of $p$, taking supremum for $ n =0,..., l$ and expectations as well as applying the Jensen inequality lead to
\begin{equation}\label{2024explicit-eq:en^p_Jensen}
\begin{aligned}
    \E\Big[ \sup_{n=0,...,l} \big\vert e_{n+1} \big\vert^{2p} \Big]
    &
    \leq
    4^{p-1}
    \E \Bigg[
        \sup_{n=0,...,l}
        \bigg\vert
        \sum_{i=0}^{n}
        (1 + 2 L h)^{n-i}
        (2 C + 1)h \vert e_i \vert^2
        \bigg\vert^p
    \Bigg]\\
    & \quad +
    4^{p-1}
    \E \Bigg[
        \sup_{n=0,...,l}
        \bigg\vert
        \sum_{i=0}^{n}
        (1 + 2 L h)^{n-i}
        2 \big\vert R_{i+1} \big\vert^2
        \bigg\vert^p
    \Bigg]\\
    & \quad
    +
    4^{p-1}
    \E \Bigg[
        \sup_{n=0,...,l}
        \bigg\vert
        \sum_{i=0}^{n}
        (1 + 2 L h)^{n-i}
        \tfrac{1}{h} \big\vert R_{i+1}^{(1)} \big\vert^2
        \bigg\vert^p
    \Bigg]\\
    & \quad +
    4^{p-1}
    \E \Bigg[
        \sup_{n=0,...,l}
        \bigg\vert
        \sum_{i=0}^{n}
        (1 + 2 L h)^{n-i}
        2 \big\langle \P(X_{t_i}) - \P(Y_i), R_{i+1}^{(2)} \big\rangle
        \bigg\vert^p
    \Bigg].
\end{aligned}
\end{equation}
Using the Jensen inequality twice again on the first and second term gives
\begin{equation}\label{2024explicit-eq:sup_en_1st_term}
\begin{aligned}
    \E \Bigg[
        \sup_{n=0,...,l}
        \bigg\vert
        \sum_{i=0}^{n}
        (1 + 2 L h)^{n-i}
        (2 C + 1) h \vert e_i \vert^2
        \bigg\vert^p
    \Bigg]
    & \leq
    \text{e}^{2Lhlp}
    \tfrac{1}{h^{p-1}}
    \E \Bigg[
        \sum_{i=0}^{l}
        (2C + 1)^p
        h^p \vert e_i \vert^{2p}
    \Bigg]\\
    & \leq
    (2C+1)^p \text{e}^{2LTp} h
    \sum_{i=0}^{l}
    \E \Big[ \sup_{k=0,...,i} \vert e_k \vert^{2p} \Big],
\end{aligned}
\end{equation}
and
\begin{equation}\label{2024explicit-eq:sup_en_2nd_term}
\begin{aligned}
    \E \Bigg[
        \sup_{n=0,...,l}
        \bigg\vert
        \sum_{i=0}^{n}
        (1 + 2 L h)^{n-i}
        2 \big\vert R_{i+1} \big\vert^2
        \bigg\vert^p
    \Bigg]
    & \leq
    2^p \text{e}^{2LTp}
    \tfrac{1}{h^{p-1}}
    \sum_{i=0}^{l}
    \E \Big[ \big\vert R_{i+1} \big\vert^{2p} \Big].
\end{aligned}
\end{equation}
Similarly for the third term,
\begin{equation}\label{2024explicit-eq:sup_en_3rd_term}
\begin{aligned}
    \E \Bigg[
        \sup_{n=0,...,l}
        \bigg\vert
        \sum_{i=0}^{n}
        (1 + Ch)^{n-i}
        \tfrac{1}{h} \big\vert R_{i+1}^{(1)} \big\vert^2
        \bigg\vert^p
    \Bigg]
    & \leq
    \text{e}^{2LTp}
    \tfrac{1}{h^{2p-1}}
    \sum_{i=0}^{l}
    \E \Big[ \big\vert R_{i+1}^{(1)} \big\vert^{2p} \Big].
\end{aligned}
\end{equation}
For the last term, observing that
$\P(X_{t_i}) - \P(Y_i)$
and
$R_{i+1}^{(2)}$ are independent for all $i \in \{0,1,...,M-1\}$,
utilizing Theorem \ref{2024explicit-thm:discrete_BDG} and \eqref{2024explicit-eq:Ph_contractivity} results in
\begin{equation*}
\begin{aligned}
    & \E \Bigg[
        \sup_{n=0,...,l}
        \bigg\vert
        \sum_{i=0}^{n}
        (1 + 2 L h)^{n-i}
        2 \big\langle \P(X_{t_i}) - \P(Y_i), R_{i+1}^{(2)} \big\rangle
        \bigg\vert^p
    \Bigg]\\
    & \leq
    2^p \text{e}^{2LTp}
    \E \Bigg[
        \bigg\vert
        \sum_{i=0}^{l}
        \Big( 
            (1 + 2 L h)^{-i}
            ( \P(X_{t_i}) - \P(Y_i) ) R_{i+1}^{(2)} 
        \Big)^2
        \bigg\vert^{\frac{p}{2}}
    \Bigg]\\
    & \leq
    2^p \text{e}^{2LTp}
    \E \Bigg[
        \bigg\vert
        \sum_{i=0}^{l}
        \big\vert e_i \big\vert^{2}
        \big\vert  R_{i+1}^{(2)} \big\vert^2
        \bigg\vert^{\frac{p}{2}}
    \Bigg]\\
    & \leq
    2^p \text{e}^{2LTp}
    \E \Bigg[
        \sup_{n=0,...,M}{ \big\vert e_n \big\vert^{p}}
        \bigg\vert
        \sum_{i=0}^{l}
        \big\vert  R_{i+1}^{(2)} \big\vert^2
        \bigg\vert^{\frac{p}{2}}
    \Bigg].
\end{aligned}
\end{equation*}
The Cauchy-Schwarz inequality infers that
\begin{equation*}
\begin{split}
    & 2^p \text{e}^{2LTp}
    \E \Bigg[
        \sup_{n=0,...,M}{ \big\vert e_n \big\vert^{p}}
        \bigg\vert
        \sum_{i=0}^{l}
        \big\vert  R_{i+1}^{(2)} \big\vert^2
        \bigg\vert^{\frac{p}{2}}
    \Bigg]\\
    & \leq
    2^p \text{e}^{2LTp}
    \Bigg( \E \bigg[
        \sup_{n=0,...,M}{ \big\vert e_n \big\vert^{2p}}
    \bigg] \Bigg)^{\frac12}
    \cdot
    \Bigg( \E \bigg[
        \bigg\vert
        \sum_{i=0}^{l}
        \big\vert  R_{i+1}^{(2)} \big\vert^2
        \bigg\vert^{p}
    \bigg] \Bigg)^{\frac12}.
\end{split}
\end{equation*}
Another application of the Jensen inequality finally leads to
\begin{equation}\label{2024explicit-eq:sup_en_4th_term}
\begin{aligned}
    & \E \Bigg[
        \sup_{n=0,...,l}
        \bigg\vert
        \sum_{i=0}^{n}
        (1 + 2 L h)^{n-i}
        2 \big\langle \P(X_{t_i}) - \P(Y_i), R_{i+1}^{(2)} \big\rangle
        \bigg\vert^p
    \Bigg]\\
    & \leq
    2^p \text{e}^{2LTp}
    \Bigg( \E \bigg[
        \sup_{n=0,...,M}{ \big\vert e_n \big\vert^{2p}}
    \bigg] \Bigg)^{\frac12}
    \cdot
    \Bigg( 
    \tfrac{1}{h^{p-1}}
    \sum_{i=0}^{l}
    \E \bigg[
        \big\vert  R_{i+1}^{(2)} \big\vert^{2p}
    \bigg] \Bigg)^{\frac12}.
\end{aligned}
\end{equation}
Combining \eqref{2024explicit-eq:sup_en_1st_term}-\eqref{2024explicit-eq:sup_en_4th_term}, one concludes that
\begin{equation}\label{2024explicit-eq:en_conclude_before_Gronwall}
\begin{aligned}
    \E\Big[ \sup_{n=0,...,l} \big\vert e_{n+1} \big\vert^{2p} \Big]
    & \leq
    4^{p-1} (2C+1)^p \text{e}^{2LTp} h
    \sum_{i=0}^{l}
    \E \Big[ \sup_{k=0,...,i} \vert e_k \vert^{2p} \Big]\\
    & \quad
    +
    4^{p-1} \text{e}^{2LTp}
    \tfrac{1}{h^{p-1}}
    \sum_{i=0}^{l}
    \E \Big[ \big\vert R_{i+1} \big\vert^{2p} \Big]\\
    & \quad
    +
    4^{p-1} \text{e}^{2LTp}
    \tfrac{1}{h^{2p-1}}
    \sum_{i=0}^{l}
    \E \Big[ \big\vert R_{i+1}^{(1)} \big\vert^{2p} \Big]\\
    & \quad
    +
    4^{p-1} 2^p \text{e}^{2LTp}
    \Bigg( \E \bigg[
        \sup_{n=0,...,M}{ \big\vert e_n \big\vert^{2p}}
    \bigg] \Bigg)^{\frac12}
    \cdot
    \Bigg( 
    \tfrac{1}{h^{p-1}}
    \sum_{i=0}^{l}
    \E \bigg[
        \big\vert  R_{i+1}^{(2)} \big\vert^{2p}
    \bigg] \Bigg)^{\frac12}\\
    & \leq
    C h
    \sum_{i=0}^{l}
    \E \Big[ \sup_{k=0,...,i} \vert e_k \vert^{2p} \Big]\\
    & \quad
    +
    C
    \tfrac{1}{h^{p}}
    \sup_{n=1,...,M}
    \E \Big[ \big\vert R_{n}^{(2)} \big\vert^{2p} \Big]
    +
    C
    \tfrac{1}{h^{2p}}
    \sup_{n=1,...,M}
    \E \Big[ \big\vert R_{n}^{(1)} \big\vert^{2p} \Big]\\
    & \quad
    +
    C
    \Bigg( \E \bigg[
        \sup_{n=0,...,M}{ \big\vert e_n \big\vert^{2p}}
    \bigg] \Bigg)^{\frac12}
    \cdot
    \Bigg( 
    \tfrac{1}{h^{p}}
    \sup_{n=1,...,M}
    \E \bigg[
        \big\vert  R_{n}^{(2)} \big\vert^{2p}
    \bigg] \Bigg)^{\frac12},
\end{aligned}
\end{equation}
where the Jensen inequality was used again to split the term $\vert R_{i+1} \vert^{2p}$.
By setting
\begin{align*}
    \delta 
    & :=
    C
    \tfrac{1}{h^{p}}
    \sup_{n=1,...,M}
    \E \Big[ \big\vert R_{n}^{(2)} \big\vert^{2p} \Big]
    +
    C
    \tfrac{1}{h^{2p}}
    \sup_{n=1,...,M}
    \E \Big[ \big\vert R_{n}^{(1)} \big\vert^{2p} \Big],\\
    \zeta 
    & := C h,\\
    \eta 
    & := 
    C 
    \Bigg( 
        \tfrac{1}{h^{p}}
        \sup_{n=1,...,M}
        \E \bigg[
            \big\vert  R_{n}^{(2)} \big\vert^{2p}
        \bigg] 
    \Bigg)^{\frac12},
\end{align*}
inequality \eqref{2024explicit-eq:en_conclude_before_Gronwall} reads as
\begin{equation}
\begin{aligned}
    \E\Big[ \sup_{n=0,...,l} \big\vert e_{n+1} \big\vert^{2p} \Big]
    & \leq
    \delta +
    \zeta
    \sum_{i=0}^{l}
    \E \Big[ \sup_{k=0,...,i} \vert e_k \vert^{2p} \Big]
    +
    \eta
    \Bigg( \E \bigg[
        \sup_{n=0,...,M}{ \big\vert e_n \big\vert^{2p}}
    \bigg] \Bigg)^{\frac12}.
\end{aligned}
\end{equation}
Finally, Lemma \ref{2024explicit-lem:Rn_bounds} together with Lemma \ref{2024explicit-lem:Gronwall} gives
\begin{equation}
\begin{aligned}
    \E\Big[ \sup_{n=0,...,M} \big\vert e_{n} \big\vert^{2p} \Big]
    & \leq
    2 \big( \delta + \eta^2 \big)
    \exp{(2 \zeta M)},
\end{aligned}
\end{equation}
which completes the proof.
\end{proof}

After establishing the upper error bound for the proposed scheme \eqref{2024explicit-eq:scheme} applied to the transformed SDE \eqref{2024explicit-eq:transformed_SDE}, it remains to perform an inverse transformation and derive the convergence results for the original SDE \eqref{2024explicit-eq:original_SDE}.
The specific form of the transformation \eqref{2024explicit-eq:Lamperti_transformation} and its inverse inherently depend on the structure of the original diffusion coefficient,
rendering them vary from models.
Nevertheless, for most financial applications, the transformation takes the canonical form:
$$
\mathbb{L}(x) = C x^k, 
\quad x>0,
$$
where $k \in \R \setminus \{0\}$.
The inverse transformation is therefore given by
$$
\mathbb{L}^{-1}(x) = C x^{\frac{1}{k}},
\quad x>0.
$$
The error after transforming back is expressed as
\begin{equation*}
\begin{split}
    \E \bigg[ 
        \sup_{n=0,...,M} 
        \Big\vert 
            \mathbb{X}_{t_n} - \mathbb{Y}_n 
	\Big\vert^{2 p} 
    \bigg]
    & =
    \E \bigg[ 
        \sup_{n=0,...,M} 
        \Big\vert 
	      \mathbb{L}^{-1}(X_{t_n})
            - 
            \mathbb{L}^{-1}(Y_n) 
	\Big\vert^{2 p} 
    \bigg]
\end{split}
\end{equation*}
The mean value theorem implies that
$$ 
\Big\vert 
    \mathbb{L}^{-1}(x) 
    - 
    \mathbb{L}^{-1}(y)
\Big\vert 
\leq
C \left( x^{\frac{1}{k}-1} + y^{\frac{1}{k}-1} \right)  
\vert x-y \vert,
\qquad x,y >0. 
$$
Thus, when $\frac{1}{k} \geq 1$, i.e., $k \in (0,1]$, the error bound relies on the moment bounds of the analytical and numerical solutions of the transformed SDE.
In contrast, when $\frac{1}{k} < 1$, i.e., $k \notin [0,1]$, the error involves the inverse moment bounds.
To this end, we provide the estimation for the inverse moments of the numerical solution \eqref{2024explicit-eq:scheme} in the next lemma.
\begin{lemma}\label{2024explicit-lem:inverse_moment_bounds_Y}
Let Assumptions 
\ref{2024explicit-ass:solution_existence_and_uniqueness}, 
\ref{2024explicit-ass:mu_one_sided_Lipschitz}, 
\ref{2024explicit-ass:inverse_moment_boundedness_of_analytical_solution},
\ref{2024explicit-ass:hat_mu_one_sided_Lipschitz}
and
\ref{2024explicit-ass:Ph_properties}
stand.
There exists a constant $C>0$ independent of the time step size $h$ such that for any $1 \leq  p \leq - \tfrac{p^*}{2 \deg^-(\mu)} \wedge \tfrac{p^*}{2m_1}$,
\begin{equation}\label{2024explicit-eq:inverse_moment_bounds_Y}
\begin{split}
		\E \bigg[
	      \sup_{n = 0,...,M}
            \big\vert Y_n \big\vert^{-2p}
        \bigg]
        & \leq
        C \tfrac{1}{h^{2 p - 1}}
        \Bigg(
        \sup_{n = 1,...,M}
        \E \Big[ \big\vert R_{n}^{(1)} \big\vert^{2 p} \Big]
        +
        \sup_{n = 1,...,M}
        \E \Big[ \big\vert R_{n}^{(2)} \big\vert^{2 p} \Big]
        \Bigg)\\
        & \quad
        +
        C \E \bigg[
	      \sup_{t \in [0,T]}
            \big\vert X_t \big\vert^{-2p}
        \bigg].
\end{split}
\end{equation}
\end{lemma}
\begin{proof}
Rewrite the scheme \eqref{2024explicit-eq:scheme} as
	\begin{align*}
		Y_{n+1}^{-1}
		& =
		\tfrac{1}{c_{-1} h} \Big[
		Y_{n+1} - \P(Y_n)
        -
        \hat\mu(\P(Y_n)) h
		-
		\sigma \Delta W_n
		\Big].
	\end{align*}
Correspondingly, rewrite the discrete form of the transformed SDE \eqref{2024explicit-eq:SDE_rewrite} as
	\begin{align*}
		X_{t_{n+1}}^{-1}
		& =
		\tfrac{1}{c_{-1} h} \Big[
		X_{t_{n+1}} - \P(X_{t_{n}})
        -
        \hat\mu(\P(X_{t_n})) h
		-
		\sigma \Delta W_n
		-
		R_{n+1}
		\Big].
	\end{align*}
By subtracting the above two equations, one obtains
\begin{equation}
\begin{split}
    X_{t_{n+1}}^{-1} - Y_{n+1}^{-1}
    =
    \tfrac{1}{c_{-1} h} \Big[
		e_{n+1}
        -
        \big( \P(X_{t_n}) - \P(Y_n) \big)
        -
        \big( \hat\mu(\P(X_{t_n})) - \hat\mu(\P(Y_n)) \big) h
        -
        R_{n+1}
	\Big].
\end{split}
\end{equation}
Raise both sides to the power of $2p\,(p \geq 1)$.
The Jensen inequality, inequality \eqref{2024explicit-eq:Ph_contractivity} and Assumption \ref{2024explicit-ass:hat_mu(Ph)_bound} infer that
\begin{equation}
\begin{split}
    \Big\vert X_{t_{n+1}}^{-1} - Y_{n+1}^{-1} \Big\vert^{2p}
    & \leq
    C \tfrac{1}{h^{2p}} \vert e_{n+1} \vert^{2p}
    +
    C \big\vert \hat\mu(\P(X_{t_n})) - \hat\mu(\P(Y_n)) \big\vert^{2p}
    +
    C \tfrac{1}{h^{2p}} \big\vert R_{n+1} \big\vert^{2p}\\
    & \leq 
    C \tfrac{1}{h^{2p}} \vert e_{n+1} \vert^{2p}
    +
    C \tfrac{1}{h^{p}} \vert e_n \vert^{2p}
    +
    C \tfrac{1}{h^{2p}} \big\vert R_{n+1} \big\vert^{2p}.
\end{split}
\end{equation}
Taking supremum and expectation as well as applying Theorem \ref{2024explicit-thm:main_convergence_result} lead to
\begin{align*}
    \E \Bigg[
	\sup_{n = 0,...,M}
    \Big\vert X_{t_{n}}^{-1} - Y_{n}^{-1} \Big\vert^{2p}
    \Bigg]
    & \leq 
	C \tfrac{1}{h^{2p}}  
    \E \bigg[
    \sup_{n = 0,...,M} 
    \vert e_n \vert^{2 p}
    \bigg]
    +
    C \tfrac{1}{h^{2 p}}
    \E \bigg[
    \sup_{n = 1,...,M} 
    \big\vert R_{n} \big\vert^{2p}
    \bigg]\\
    & \leq
    C \tfrac{1}{h^{2 p}}
    \sup_{n = 1,...,M}
    \E \Big[ \big\vert R_{n}^{(1)} \big\vert^{2 p} \Big]
    +
    \tfrac{1}{h^{p}}
    \sup_{n = 1,...,M}
    \E \Big[
    \big\vert  R_{n}^{(2)} \big\vert^{2 p}
    \Big]\\
    & \quad
    +
    C \tfrac{1}{h^{2 p}}
    \E \bigg[
    \sup_{n = 1,...,M} 
    \big\vert R_{n}^{(1)} \big\vert^{2 p}
    \bigg]
    +
    C \tfrac{1}{h^{2 p}}
    \E \bigg[
    \sup_{n = 1,...,M} 
    \big\vert R_{n}^{(2)} \big\vert^{2 p}
    \bigg],
\end{align*}
for $1 \leq p \leq - \tfrac{p^*}{2 \deg^-(\mu)} \wedge \tfrac{p^*}{2m_1}$, where the term $\sup\limits_{n = 1,...,M} \vert R_{n} \vert^{2p}$ is splitted again using the Jensen inequality.
Observing
$\sup\limits_{ 0\leq n \leq M} \vert a_n \vert \leq \sum\limits_{n=0}^M \vert a_n \vert$
one obtains
\begin{align*}
    \E \Bigg[
	\sup_{n = 0,...,M}
    \Big\vert X_{t_{n}}^{-1} - Y_{n}^{-1} \Big\vert^{2p}
    \Bigg]
    & \leq
    C \tfrac{1}{h^{2 p - 1}}
    \Bigg(
    \sup_{n = 1,...,M}
    \E \Big[ \big\vert R_{n}^{(1)} \big\vert^{2 p} \Big]
    +
    \sup_{n = 1,...,M}
    \E \Big[ \big\vert R_{n}^{(2)} \big\vert^{2 p} \Big]
    \Bigg).
\end{align*}
The assertion follows now by the triangle inequality.
\end{proof}
The following lemma is provided to bound $\E \Big[
	      \sup_{t \in [0,T]}
            \big\vert X_t \big\vert^{-2p}
        \Big]$.
The proof is a slight modification of \cite[Lemma 2.13]{neuenkirch2014first} and follows the same lines as the original, so we omit the proof.
\begin{lemma}\label{2024explicit-thm:inverse_moment_bound_X}
Let Assumptions 
\ref{2024explicit-ass:solution_existence_and_uniqueness}, 
\ref{2024explicit-ass:mu_one_sided_Lipschitz}
and
\ref{2024explicit-ass:inverse_moment_boundedness_of_analytical_solution}
stand.
For any $p \geq 1$, there exists a constant $C_p>0$ such that
		\begin{align}
			\E \bigg[
			\sup_{t \in [0,T]}
			\vert X_t\vert ^{-2 p}
			\bigg]
			\leq
			C_p\left( 1+ \sup_{t \in [0,T]} \E \bigg[
			\vert X_t\vert ^{-(2 p + 2)}
			\bigg] \right).
		\end{align}
	\end{lemma}

\section{Applications to financial models} \label{2024explicit-sec:examples}

In this section, we apply our scheme to a variety of financial models. 
The analysis will proceed as follows:
\begin{enumerate}
    \item Apply the appropriate Lamperti-type transformation to the model to obtain the transformed SDE.
    \item Verify whether the assumptions imposed on the SDE model are satisfied.
    \item Choose a proper correction function $\P$ and verify that it satisfies Assumptions \ref{2024explicit-ass:Ph_properties} and \ref{2024explicit-ass:hat_mu(Ph)_bound}.
    Apply the scheme \eqref{2024explicit-eq:scheme} to the transformed SDE. 
    Estimate the remainder terms $R_n$ and present convergence results with respect to the transformed SDE.
    \item Apply the inverse Lamperti transformation and derive the convergence results for the original SDE model.
\end{enumerate}
In the second step, specifically, Assumption \ref{2024explicit-ass:solution_existence_and_uniqueness} will be verified using the classical Feller test; 
Assumption \ref{2024explicit-ass:mu_one_sided_Lipschitz} will be confirmed with the help of Lemma \ref{2024explicit-lem:FLP_properties}; 
Assumption \ref{2024explicit-ass:inverse_moment_boundedness_of_analytical_solution} can be validated based on the results from existing literature.

\subsection{CIR process}

The Cox-Ingersoll-Ross (CIR) process, described by
\begin{equation}
\label{2024explicit-eq:CIR_original}
\begin{cases}
    \d \mathbb{X}_t 
        = 
        \kappa (\theta - \mathbb{X}_t) \, \d t 
        + 
        \hat\sigma \sqrt{\mathbb{X}_t} \, \d W_t,
        \quad
        t \in (0, T],\\
        \mathbb{X}_0 = x_0 > 0,
\end{cases}
\end{equation}
first introduced by Feller \cite{Feller}
and further developed by Cox, Ingersoll and Ross \cite{cox2005theory}, is now widely used in financial modeling.
By applying the Lamperti-type transformation
$\mathbb{L}: (0,+\infty) \rightarrow (0,+\infty)$
of the form
$\mathbb{L}(x) := 2 \sqrt{x}$,  
we obtain the following transformed SDE:
\begin{align} \label{eq:CIR_transformed}
    \d X_t
    =
    \mu(X_t) \d t + \sigma \d W_t, 
\end{align}
where $\sigma=\hat\sigma$ and
\begin{align}\label{2024explicit-eq:CIR_mu_definition}
    \mu(x)
    & =
    2 \big(
       \kappa \theta
        -
        \tfrac{\sigma^2 }{4}
    \big)
    x^{-1}
    -  
    \tfrac{\kappa}{2} x.
\end{align}
Recall the definition of $\hat\mu$ in \eqref{2024explicit-eq:hat_mu_definition} and Remark \ref{2024explicit-remark_c-1}.
Setting
$$c_{-1}
    =
    2 \big(
       \kappa \theta
        -
        \tfrac{\sigma^2 }{4}
    \big),$$
gives
\begin{equation}\label{2024explicit-eq:CIR_hat_mu_c-1_definition}
    \hat\mu(x)
    =
    -  
    \tfrac{\kappa}{2} x,
\end{equation}
which satisfies Assumption \ref{2024explicit-eq:hat_mu_one_sided_Lipschitz}.
Note that $\mu$ and $\hat\mu$ are FLPs with 
\begin{align*}
    \deg^{-}(\mu) &= -1,&
    \deg^{+}(\mu) &= 1,\\
    \deg^{-}(\hat\mu) &= 1,&
    \deg^{+}(\hat\mu) &= 1,\\
    \coeff^{-}(\mu) &= 
\tfrac{1}{2}
\big(
    4 \kappa \theta -  \sigma^2
\big),&
    \coeff^{+}(\mu) &= -\tfrac{\kappa}{2}.
\end{align*}
Next we verify the assumptions imposed on the SDE model. 
\begin{enumerate}[(1)]
    \item The Feller test ensures that Assumption \ref{2024explicit-ass:solution_existence_and_uniqueness} holds with $2 \kappa \theta \geq  \hat\sigma^2$.

    \item Under the condition $2 \kappa \theta \geq \sigma^2 $, Assumption \ref{2024explicit-ass:mu_one_sided_Lipschitz} can be easily confirmed by observing that
    \begin{align*}
        \mu'(x) 
        =
        - 2( \kappa \theta - \tfrac{\sigma^2}{4} ) x^{-2}
        - \tfrac{\kappa}{2}
        < 0.
    \end{align*}

    \item It is known from \cite{dereich2012euler} that the CIR process satisfies
    \begin{align}\label{cir-inv-moments}
        \sup_{t \in [0,T]} 
        \mathbb{E} \big[ \vert \mathbb{X}_t \vert^p \big]
        <
        + \infty 
        \quad
        \textrm{if} 
        \quad 
        p 
        > 
        -\frac{2\kappa \theta}{\hat\sigma^2}. 
    \end{align} 
    Consequently, one has
    \begin{align} \label{2024explicit-eq:cir-inv-moments_lamperti}
        \sup_{t \in [0,T]} 
        \mathbb{E} \Big[ \vert {X}_t \vert^{p} \Big]
        <
        +\infty
        \quad 
        \textrm{if}
        \quad 
        p > -\frac{4\kappa \theta}{\hat\sigma^2} .
    \end{align} 
    Thus Assumption \ref{2024explicit-ass:inverse_moment_boundedness_of_analytical_solution} holds with $p^* = \tfrac{4 \kappa \theta}{\hat\sigma^2}- \epsilon$ for an arbitrarily small positive constant $\epsilon$.
    Note that $p^*>2 = - 2 \deg^-(\mu)$, which allows $p$ to live in $[1,p^*]$, thereby satisfying a necessary condition for Lemma \ref{2024explicit-lem:Rn_bounds} and Theorem \ref{2024explicit-thm:main_convergence_result}.
\end{enumerate}

{\color{black}
Since $\hat\mu$ is linear, one may choose $\P = I$, i.e., the identity mapping, thereby ensuring that Assumptions \ref{2024explicit-ass:Ph_properties} and \ref{2024explicit-ass:hat_mu(Ph)_bound} are automatically satisfied with any constants $m_1,m_2$ (from Assumption~\ref{2024explicit-ass:Ph_properties}). In particular, setting $m_1 = 0$ removes the restriction  $p \leq \frac{p^*}{2m_1}$ in Theorem \ref{2024explicit-thm:main_convergence_result}.
}
The scheme \eqref{2024explicit-eq:scheme} now read as
\begin{equation}\label{2024explicit-eq:scheme_for_CIR}
    Y_{n+1}
    =
    Y_n
    +
    2 \big(\kappa \theta - \tfrac{\sigma^2}{4}\big) h Y_{n+1}^{-1}
    - 
    \tfrac{1}{2} \kappa h Y_n
    +
    \sigma \Delta W_n.
\end{equation}
Recall that
\begin{equation*}
\begin{aligned}
    R_{n+1}
    & =
    X_{t_n} - \P (X_{t_n})
    +
    \int_{t_n}^{t_{n+1}}
        \mu(X_s) 
    \d s
    -
    c_{-1} h X_{t_{n+1}}^{-1}
    -
    \hat \mu(\P(X_{t_n})) h\\
    & =
    \int_{t_n}^{t_{n+1}}
    \Big(
        -\tfrac{\kappa}{2} X_s + c_{-1} X_s^{-1} 
    \Big)
    \d s
    -
    c_{-1} h X_{t_{n+1}}^{-1}
    +
    \tfrac{\kappa}{2} X_{t_n} h.
\end{aligned}
\end{equation*}
{
The It\^o formula gives
\begin{equation}
\begin{aligned}
    R_{n+1}
    & =
    -\tfrac{\kappa}{2}
    \int_{t_n}^{t_{n+1}}
    \bigg(
        \int_{t_n}^s
            \mu(X_r)
        \d r
        +
        \int_{t_n}^s
            \sigma
        \d W_r
    \bigg)
    \d s\\
    & \quad
    +
    c_{-1} 
    \int_{t_n}^{t_{n+1}}
    \int_{t_{n+1}}^s
    \Big(
        - X_r^{-2}
        \mu(X_r)
        +
        \sigma^2
        X_r^{-3}
    \Big)
    \d r
    \d s
    +
    c_{-1}
    \int_{t_n}^{t_{n+1}}
    \int_{t_{n+1}}^s
        - X_r^{-2}
        \sigma
    \d W_r
    \d s\\
    & =
    \underbrace{
    -\tfrac{\kappa}{2}
    \int_{t_n}^{t_{n+1}}
        \int_{t_n}^s
            \mu(X_r)
        \d r
    \d s
    +
    c_{-1} 
    \int_{t_n}^{t_{n+1}}
    \int_{t_{n+1}}^s
    \Big(
        - X_r^{-2}
        \mu(X_r)
        +
        \sigma^2
        X_r^{-3}
    \Big)
    \d r
    \d s
    }_{: = R_{n+1}^{(1)}}\\
    & \quad
    \underbrace{
    - \tfrac{\kappa}{2}
    \int_{t_n}^{t_{n+1}}
        \int_{t_n}^s
            \sigma
        \d W_r
    \d s
    +
    c_{-1}
    \int_{t_n}^{t_{n+1}}
    \int_{t_{n+1}}^s
        - X_r^{-2}
        \sigma
    \d W_r
    \d s
    }_{: = R_{n+1}^{(2)}}.
\end{aligned}
\end{equation}
}
The Jensen inequality together with the H\"older inequality reveals that for $p\geq 1$, 
\begin{equation}
\begin{aligned}
    \E \Big[ \big\vert R_{n+1}^{(1)} \big\vert^{2p} \Big]
    & \leq
    2^{2p-1}
    \E\Bigg[
    \bigg\vert 
        -\tfrac{\kappa}{2}
        \int_{t_n}^{t_{n+1}}
        \int_{t_n}^s
            \mu(X_r)
        \d r
        \d s
    \bigg\vert^{2p}
    \Bigg]\\
    & \quad +
    2^{2p-1}
    \E\Bigg[
    \bigg\vert 
        c_{-1} 
        \int_{t_n}^{t_{n+1}}
        \int_{t_{n+1}}^s
        \Big(
            - X_r^{-2}
            \mu(X_r)
            +
            \sigma^2
            X_r^{-3}
        \Big)
        \d r
        \d s
    \bigg\vert^{2p}
    \Bigg]\\
    & \leq
    \half \kappa^{2p} h^{4p-2}
    \E\Bigg[
        \int_{t_n}^{t_{n+1}}
        \int_{t_n}^s
            \big\vert
            \mu(X_r)
            \big\vert^{2p}
        \d r
        \d s
    \Bigg]\\
    & \quad +
    2^{2p-1} c_{-1}^{2p} h^{4p - 2}
    \E\Bigg[
        \int_{t_n}^{t_{n+1}}
        \int_{t_{n+1}}^s
        \Big\vert
            - X_r^{-2}
            \mu(X_r)
            +
            \sigma^2
            X_r^{-3}
        \Big\vert^{2p}
        \d r
        \d s
    \Bigg].
\end{aligned}
\end{equation}
Thus Lemma \ref{2024explicit-lem:X_integrability} and  \eqref{2024explicit-eq:cir-inv-moments_lamperti} imply that for any $6 p \leq p^* = \tfrac{4 \kappa \theta}{\hat\sigma^2} - \epsilon$, i.e., $1 \leq p < \tfrac{2 \kappa \theta}{3 \hat\sigma^2}$, one obtains
\begin{align}\label{2024explicit-eq:CIR_Rn_1_estimate}
    \E\Big[\big\vert R_{n+1}^{(1)} \big\vert^{2p}\Big]
    & \leq
    C h^{4p} 
    \cdot
    \sup_{0 \leq r \leq T}
    \E\Big[\big\vert X_r \big\vert^{\min\{-6 p,\,  2 p \deg^-(\mu)\}}\Big]\\
    \nonumber
    & \leq
    C h^{4p}.
\end{align}
Concerning $R_{n+1}^{(2)}$, the Jensen inequality implies
\begin{equation*}
\begin{aligned}
    \E\bigg[ \big\vert R_{n+1}^{(2)} \big\vert^{2p} \bigg]
    & \leq
    2^{2p - 1}
    \E \Bigg[
    \bigg\vert
        \int_{t_n}^{t_{n+1}}
        \int_{t_n}^s
        \sigma
        \d W_r
        \d s
    \bigg\vert^{2p}
    \Bigg]\\
    & \quad
    +
    2^{2p - 1}
    \E \Bigg[
    \bigg\vert
        c_{-1}
        \int_{t_n}^{t_{n+1}}
        \int_{t_{n+1}}^s
        - X_r^{-2} \sigma
        \d W_r
        \d s
    \bigg\vert^{2p}
    \Bigg].
\end{aligned}
\end{equation*}
Applying the H\"older inequality as well as the moment inequality \cite[Theorem 7.1]{mao2007stochastic} results in
\begin{equation*}
\begin{aligned}
    \E\Big[ \big\vert R_{n+1}^{(2)} \big\vert^{2p} \Big]
    & \leq
    2^{2p - 1} h^{3p - 2} C_p
    \E \Bigg[
        \int_{t_n}^{t_{n+1}}
        \int_{t_n}^s
        \sigma^{2p}
        \d r
        \d s
    \Bigg]\\
    & \quad
    +
    2^{2p - 1} c_{-1}^{2p} h^{3p-2} C_p
    \E \Bigg[
        \int_{t_n}^{t_{n+1}}
        \int_{t_{n+1}}^s
        \big\vert
        - X_r^{-2} \sigma
        \big\vert^{2p}
        \d r
        \d s
    \Bigg].
\end{aligned}
\end{equation*}
Thus  Assumption \ref{2024explicit-ass:inverse_moment_boundedness_of_analytical_solution} implies that for any $4p \leq p^* = \tfrac{4 \kappa \theta}{\hat\sigma^2} - \epsilon$, i.e., $1 \leq p < \tfrac{ \kappa \theta}{\hat\sigma^2} $, 
one obtains
\begin{align}\label{2024explicit-eq:CIR_Rn_2_estimate}
    \E\Big[
    \big\vert R_{n+1}^{(2)} \big\vert^{2p}
    \Big]
    & \leq
    C h^{3p} 
    \cdot
    \sup_{0 \leq r \leq T}
    \E\Big[\big\vert X_r \big\vert^{- 4 p}\Big]\\
    \nonumber
    & \leq
    C h^{3p}.
\end{align}
Furthermore, Theorem \ref{2024explicit-thm:main_convergence_result} gives for $1 \leq p \leq - \tfrac{p^*}{2 \deg^-(\mu)} \wedge \tfrac{p^*}{2m_1} = \tfrac{2\kappa \theta}{\hat\sigma^2} - \epsilon$
\begin{equation*}
    \begin{split}
        \E\Big[ \sup_{n=0,...,M} \big\vert e_{n} \big\vert^{2p} \Big]
        & \leq
        2 C \text{e}^{2T}
        \Bigg(
            \tfrac{1}{h^{2p}}
            \sup_{n = 1,...,M}
            \E \Big[ \big\vert R_{n}^{(1)} \big\vert^{2p} \Big]
            +
            \tfrac{1}{h^{p}}
            \sup_{n = 1,...,M}
            \E \bigg[
            \big\vert  R_{n}^{(2)} \big\vert^{2p}
            \bigg] 
        \Bigg).
    \end{split}
\end{equation*}
Therefore, combining \eqref{2024explicit-eq:CIR_Rn_1_estimate} and \eqref{2024explicit-eq:CIR_Rn_2_estimate} one obtains for $1 \leq p < \tfrac{2 \kappa \theta}{3 \hat\sigma^2}$
\begin{equation}
    \begin{split}
        \E\Big[ \sup_{n=0,...,M} \big\vert e_{n} \big\vert^{2p} \Big]
        & \leq
        C h^{2p}
    \end{split}.
\end{equation}
Finally, with the aid of Lemma \ref{2024explicit-lem:X_integrability} and Lemma \ref{2024explicit-lem:Y_integrability}, transforming back yields that for any $1 \leq p < \tfrac{2 \kappa \theta}{3 \hat\sigma^2}$,
\begin{align*}
\nonumber
    & \E \Big[
        \sup_{n = 0,...,M} 
        \big\vert  \mathbb{X}_{t_n} - \mathbb{Y}_n \big\vert ^{2 p}
    \Big]
    = \frac{1}{4^{2p}}
    \E \Big[
        \sup_{n = 0,...,M} 
        \big\vert  
           X_{t_n}^2 
            - 
            Y_n^2 
        \big\vert ^{2 p}
    \Big]\\
\nonumber
    & \leq
    C \cdot \Bigg( \E \Big[
        \sup_{n = 0,...,M} 
        \big\vert  X_{t_n} + Y_n \big\vert ^{2 p\tfrac{1 + \epsilon}{\epsilon}}
    \Big] \Bigg)^{\tfrac{\epsilon}{1 + \epsilon}}
    \times
    \Bigg( \E \Big[
        \sup_{n = 0,...,M} 
        \big\vert  X_{t_n} - Y_n \big\vert ^{2 p(1 + \epsilon)}
    \Big] \Bigg)^{\tfrac{1}{1 + \epsilon}}\\
    & \leq 
    C h^{2 p}.
\end{align*}
\begin{proposition}\label{2024explicit-prop:CIR}
Let $2 \kappa \theta > \hat\sigma^2$.
For any $1 \leq p < \tfrac{2 \kappa \theta}{3 \hat\sigma^2}$, the proposed scheme is $2p$-strongly convergent with order $1$ for the CIR process.
\end{proposition}
We mention that,  as indicated by Proposition \ref{2024explicit-prop:CIR}, for the CIR model the proposed scheme achieves the first-order convergence under the same conditions as the Lamperti backward Euler-Maruyama (LBEM) scheme in \cite{neuenkirch2014first}. 

\subsection{Heston-3/2 volatility}

The Heston-3/2 volatility model, as introduced in \cite{heston1997simple}, is given by
\begin{align}\label{2024explicit-eq:Heston_original}
\begin{cases}
    \d \mathbb{X}_t
    =
    a_1 \mathbb{X}_t \big( a_2 - \mathbb{X}_t \big) \d t
    +
    a_3 \mathbb{X}_t ^{\frac{3}{2}} \, \d W_t,
    \quad
    t \in (0, T],\\
    \mathbb{X}_0 = x_0 > 0,
\end{cases}
\end{align}
where $a_1, a_2, a_3>0$.  
The model can be derived as the inverse of the CIR process and serves as an extension of the original Heston model.
By applying the Lamperti-type transformation
$\mathbb{L}: (0,+\infty) \rightarrow (0,+\infty)$
of the form
$\mathbb{L}(x) := 2x^{-\tfrac{1}{2}}$,  
one obtains the following SDE:
\begin{align} \label{eq:Heston_transformed}
    \d X_t
    =
    \mu(X_t) \d t + \sigma \d W_t, 
\end{align}
where $\sigma= - a_3$ and
\begin{align}\label{2024explicit-eq:Heston-3/2_mu_definition}
        \mu(x)
        =
        2(a_1 + \tfrac34 a_3^2) x^{-1}
        -
        \tfrac{a_1 a_2}{2} x.
\end{align}
As investigated in \cite{drimus2012options}, the Feller non-explosion condition is satisfied for $a_1 > 0$.
Note that \eqref{2024explicit-eq:Heston-3/2_mu_definition} reduces to \eqref{2024explicit-eq:CIR_mu_definition} when the parameters are set as follows:
    \begin{align}
        \tfrac{\kappa \theta}{\hat\sigma^2}
        =
        \tfrac{a_1}{a_3^2} + 1
        \quad
        \text{and}
        \quad
        \kappa = a_1 a_2.
    \end{align} 
Hence, using the same $c_{-1} = 2(a_1 + \tfrac34 a_3^2)$ and correction function $\P = I$, the results derived for the transformed CIR model \eqref{eq:CIR_transformed} can be directly applied by replacing the corresponding parameters.
More precisely,
Theorem \ref{2024explicit-thm:main_convergence_result} together with \eqref{2024explicit-eq:CIR_Rn_1_estimate} and \eqref{2024explicit-eq:CIR_Rn_2_estimate} gives
\begin{equation}\label{2024explicit-eq:Heston_en_estimate_by_X_inverse_moment}
    \begin{split}
        \E\Big[ \sup_{n=0,...,M} \big\vert e_{n} \big\vert^{2p} \Big]
        & \leq
        2 C \text{e}^{2T}
        \Bigg(
            \tfrac{1}{h^{2p}}
            \sup_{n = 1,...,M}
            \E \Big[ \big\vert R_{n}^{(1)} \big\vert^{2p} \Big]
            +
            \tfrac{1}{h^{p}}
            \sup_{n = 1,...,M}
            \E \bigg[
            \big\vert  R_{n}^{(2)} \big\vert^{2p}
            \bigg] 
        \Bigg)\\
        & \leq
        2 C \text{e}^{2T}
        \Bigg(
            C h^{2p} 
            \sup_{0 \leq r \leq T}
            \E\Big[\big\vert X_r \big\vert^{- 6 p}\Big]
            +
            C h^{2p} 
            \sup_{0 \leq r \leq T}
            \E\Big[\big\vert X_r \big\vert^{- 4 p}\Big]
        \Bigg)\\
        & \leq
        C h^{2p}
        \sup_{0 \leq r \leq T}
        \E\Big[\big\vert X_r \big\vert^{- 6 p}\Big]
    \end{split}
    \end{equation}
for any $1 \leq p < \tfrac{ 2 \kappa \theta}{ \hat\sigma^2} = 2 \Big( \tfrac{a_1}{a_3^2} + 1 \Big)$.
Transforming back yields
\begin{align*}
\nonumber
       & \E \Big[ \sup_{n = 0,...,M} \big \vert
            \mathbb{X}_{t_n} - \mathbb{Y}_n
        \big \vert^{2 p} \Big]
        \\ & \quad  =
      4^{2p}  \, \E \Bigg[ \sup_{n = 0,...,M} \bigg \vert
            X_{t_n}^{-2} 
            - 
            Y_n^{-2}
        \bigg \vert^{2 p} \Bigg]\\
\nonumber
        & \quad \leq
        C \cdot \E \Bigg[ \sup_{n = 0,...,M} 
            \Big( \big\vert  X_{t_n} \big\vert ^{-3} + \big\vert  Y_n \big\vert ^{-3} \Big)^{2 p}
            \big\vert  X_{t_n} - Y_n \big\vert ^{2 p}
        \Bigg]\\
\nonumber
        & \quad \leq
        C \cdot \Bigg( 
            \E \Big[ \sup_{n = 0,...,M} 
                \big\vert  X_{t_n} \big\vert ^{- 6 a p}
            \Big] ^{\frac1a}
            +
            \E \Big[ \sup_{n = 0,...,M} 
                \big\vert  Y_n \big\vert ^{-6 a p}
            \Big] ^{\frac1a}
        \Bigg) 
        \cdot
        \Bigg( \E \Big[ \sup_{n = 0,...,M} 
            \big\vert  X_{t_n} - Y_n \big\vert ^{2 b p}
        \Big] \Bigg) ^{\frac1b},
\end{align*}
where the H\"{o}lder inequality has been used in the last step with  $a,b > 1$ and $\tfrac1a + \tfrac1b = 1$.
For the inverse moment of the numerical solution,
Lemma \ref{2024explicit-lem:inverse_moment_bounds_Y} together with \eqref{2024explicit-eq:CIR_Rn_1_estimate} and \eqref{2024explicit-eq:CIR_Rn_2_estimate} infers that for any $1 \leq q < \tfrac{ 2 \kappa \theta}{ \hat\sigma^2} = 2 \Big( \tfrac{a_1}{a_3^2} + 1 \Big)$,
\begin{align*}
		\E \bigg[
	      \sup_{n = 0,...,M}
            \big\vert Y_n \big\vert^{-2q}
        \bigg]
        & \leq
        C \tfrac{1}{h^{2 q - 1}}
        \Bigg(
        \sup_{n = 0,...,M-1}
        \E \Big[ \big\vert R_{n+1}^{(1)} \big\vert^{2 q} \Big]
        +
        \sup_{n = 0,...,M-1}
        \E \Big[ \big\vert R_{n+1}^{(2)} \big\vert^{2 q} \Big]
        \Bigg)\\
        & \quad
        +
        C \E \bigg[
	      \sup_{t \in [0,T]}
            \big\vert X_t \big\vert^{-2q}
        \bigg]\\
        & \leq
        C
        \sup_{t \in [0,T]}
        \E \bigg[
            \big\vert X_t \big\vert^{-6q}
        \bigg]
        +
        C \E \bigg[
	      \sup_{t \in [0,T]}
            \big\vert X_t \big\vert^{-2q}
        \bigg].
	\end{align*}
Using Lemma \ref{2024explicit-thm:inverse_moment_bound_X} one can further deduce that
\begin{equation}
\begin{split}
    \E \bigg[
	      \sup_{n = 0,...,M}
            \big\vert Y_n \big\vert^{-2q}
    \bigg]
    & \leq
    C_q
    \left( 
        \sup_{t \in [0,T]}
        \E \bigg[
            \big\vert X_t \big\vert^{-6q}
        \bigg] 
        + 
        \sup_{t \in [0,T]} 
        \E \bigg[
			\vert X_t\vert ^{-(2 q + 2)}
		\bigg] 
   \right)\\
   & \leq
   C_q
        \sup_{t \in [0,T]}
        \E \bigg[
            \big\vert X_t \big\vert^{-6q}
        \bigg]
\end{split}
\end{equation}
for $1 \leq q < \tfrac{ 2 \kappa \theta}{ \hat\sigma^2} = 2 \Big( \tfrac{a_1}{a_3^2} + 1 \Big)$.
Recalling \eqref{2024explicit-eq:Heston_en_estimate_by_X_inverse_moment} one gets
\begin{equation}
\begin{aligned}
    \E& \Big[ \sup_{n=0,...,M} \big\vert \mathbb{X}_{t_n} - \mathbb{Y}_n \big\vert^{2p} \Big]\\
    & \leq
    C \Bigg(
        1 + 
        \sup_{t \in [0,T]}
        \E \Big[
            \big\vert X_t \big\vert^{-18ap}
        \Big] ^{\frac{1}{a}}
    \Bigg)
    \cdot
    \Bigg( 
        \E \Big[ \sup_{n = 0,...,M} 
            \big\vert  X_{t_n} - Y_n \big\vert ^{2 b p}
        \Big] \Bigg) ^{\frac1b}\\
    & \leq
    C \Bigg(
        1 + 
        \sup_{t \in [0,T]}
        \E \Big[
            \big\vert X_t \big\vert^{-18ap}
        \Big] ^{\frac{1}{a}}
    \Bigg)
    \cdot
    C h^{2p}
    \Bigg( 
        \sup_{t \in [0,T]}
        \E \Big[
            \big\vert X_t \big\vert^{-6bp}
        \Big] 
    \Bigg) ^{\frac1b}.
\end{aligned}
\end{equation}
By aligning the exponents to achieve the best possible values for $a$ and $b$, it follows that $a=\tfrac43$, $b = 4$.
For the CIR process, one knows by \eqref{2024explicit-eq:cir-inv-moments_lamperti} that
\begin{equation}
    \E \Big[ \sup_{n=0,...,M} \big\vert \mathbb{X}_{t_n} - \mathbb{Y}_n \big\vert^{2p} \Big]
    \leq
    C h^{2p},
    \quad
    \forall 1 \leq p < \tfrac{1}{6} \big( \tfrac{a_1}{a_3^2} + 1 \big).
\end{equation}
\begin{proposition}
\label{2024explicit-prop:32model}
For any $1 \leq p < \tfrac{1}{6}\big( \tfrac{a_1}{a_3^2} + 1 \big)$, the proposed scheme is $2p$-strongly convergent with order $1$ for the Heston-3/2 volatility model.
\end{proposition}

\subsection{CEV Process}

The constant elasticity of variance (CEV) model describes a stochastic process that provides a flexible framework for modeling asset prices, governed by the following SDE:
\begin{align}\label{2024explicit-eq:CEV_original}
    \d \mathbb{X}_t
    =
    \kappa (\theta - \mathbb{X}_t) \d t
    +
    \hat\sigma \, \mathbb{X}_t ^ {d}\, \d W_t, \quad t \in (0, T], \qquad \mathbb{X}_0 = x_0 >0,
\end{align}
where $0.5 < d < 1$ and $\kappa, \theta, \hat\sigma >0$. 
The model is introduced by Cox \cite{cox1996constant}, offering a more adaptable approach to asset price modeling. 

By applying the Lamperti-type transformation
$\mathbb{L}: (0,+\infty) \rightarrow (0,+\infty)$
of the form
$\mathbb{L}(x) := \tfrac{1}{1-d} x^{1 - d}$, 
one obtains the following SDE:
\begin{align}
        \d X_t
        =
        \mu(X_t) \d t
        +
        \sigma \d W_t
\end{align}
where $\sigma = \hat\sigma$ and
\begin{align}
        \mu(x)
        &=
            \kappa \theta \left((1 - d) x\right)^{- \frac{d}{1 - d} }
            -
            \tfrac{d \sigma^2}{2} \left((1 - d) x\right)^{-1}
            -
            \kappa \left((1 - d) x\right).
\end{align}
Recall Remark \ref{2024explicit-remark_c-1} and set $c_{-1}  = \tfrac{d \sigma^2}{2(1 - d)}$.
Thus
\begin{equation}\label{2024explicit-eq:CEV_hat_mu_c-1_definition}
    \hat\mu(x)
    =
    \mu(x) - c_{-1}x^{-1}
    =
    \kappa \theta \left((1 - d) x\right)^{- \frac{d}{1 - d} }
    -
    d \sigma^2 \left((1 - d) x\right)^{-1}
    -
    \kappa \left((1 - d) x\right).
\end{equation}
Note that $\mu$ and $\hat\mu$ are FLPs with 
\begin{align*}
    \deg^{-}(\mu) &= -\tfrac{d}{1-d},&
    \deg^{+}(\mu) &= 1,\\
    \deg^{-}(\hat\mu) &= -\tfrac{d}{1-d},&
    \deg^{+}(\hat\mu) &= 1,\\
    \coeff^{-}(\mu) &= \kappa\theta (1-d)^{-\frac{d}{1-d}},&
    \coeff^{+}(\mu) &= -\kappa (1-d).
\end{align*}
Next we verify the Assumptions imposed on the SDE model.
\begin{enumerate}
    \item The Feller test ensures that Assumption \ref{2024explicit-ass:solution_existence_and_uniqueness} holds with $D = (0,+\infty)$.

    \item Assumption \ref{2024explicit-ass:mu_one_sided_Lipschitz} can be easily confirmed by observing that
    ${\deg}^{-}(\mu) < 0 < {\deg}^{+}(\mu)$
    and
    $\coeff^+(\mu) < 0 < \coeff^-(\mu)$,
    thanks to Lemma \ref{2024explicit-lem:FLP_properties}.

    \item From \cite{berkaoui2008euler}, it is known that for any $q \in \mathbb{R}$ and $T>0$
        \begin{align}\label{2024explicit-eq:CEV_moment_bounds_original}
            \sup_{0 \leq t \leq T} \E \Big[ \vert \mathbb{X}_t \vert^q \Big]
            <
            \infty.
        \end{align}
    Consequently, one has for any $q \in \R$ and $T>0$
    \begin{align} \label{2024explicit-eq:CEV_inverse_moment_bounds_lamperti}
        \sup_{t \in [0,T]} 
        \mathbb{E} \Big[ \vert {X}_t \vert^{q} \Big]
        <
        +\infty.
    \end{align} 
    Thus Assumption \ref{2024explicit-ass:inverse_moment_boundedness_of_analytical_solution} is satisfied with $p^* = +\infty$.
\end{enumerate}
{\color{black}
Set
\begin{equation}
    \P(x) = \min\{ \max \{x , \Cs h^{\beta} \} ,\Cl h^{- \alpha}\},
\end{equation}
where $\Cs,\Cl$ are arbitrary positive constants and
\begin{equation}\label{2024explicit-eq:CEV_alpha_beta}
    \beta = \tfrac{1}{2(1 -\deg^-(\hat\mu))} =\tfrac{1-d}{2},
    \quad
    \alpha = \tfrac{1}{2\deg^+(\hat\mu)} = \half.
\end{equation}
Here, the constant $\Cs,\Cl$ can be chosen arbitrarily, but $\Cs$ should be sufficiently small and $\Cl$ sufficiently large to reduce the frequency of corrections, thereby reducing the bias.
Next, we verify whether $\P$ satisfies Assumptions \ref{2024explicit-ass:Ph_properties} and \ref{2024explicit-ass:hat_mu(Ph)_bound}.
Assertion \eqref{2024explicit-eq:Ph_contractivity} follow directly from the definition. 
Note that
\begin{equation}
\begin{aligned}
    \big\vert \P(x) - x \big\vert
    & =
    \mathbbm{1}_{\{ 0< x < \Cs h^\beta\}}
    \big\vert \P(x) - x \big\vert
    +
    \mathbbm{1}_{\{ x >  \Cl h^{-\alpha} \}}
    \big\vert \P(x) - x \big\vert.
\end{aligned}
\end{equation}
For the first term, since $0<\beta<2$,
$$
\mathbbm{1}_{\{ 0< x < \Cs h^\beta\}}
    \big\vert \P(x) - x \big\vert
    \leq
    \Cs h^{\beta}
    \leq
    (x^{-1} \Cs h^{\beta})^{\frac{2-\beta}{\beta}} 
    \cdot 
    \Cs h^{\beta}
    =
    \Cs^{\frac{2}{\beta}}h^2 \cdot x^{-\frac{2-\beta}{\beta}}.
$$
For the second term,
$$
\mathbbm{1}_{\{ x > \Cl h^{-\alpha} \}}
    \big\vert \P(x) - x \big\vert
    \leq
    x
    \leq
    x \cdot (x \Cl^{-1} h^{\alpha})^{\frac{2}{\alpha}}
    =
    \Cl^{-\frac{2}{\alpha}}h^2 \cdot x^{\frac{2 + \alpha}{\alpha}}.
$$
Therefore, Assertion \eqref{2024explicit-eq:correction_error} holds with 
$$m_1 = \tfrac{2-\beta}{\beta} = \tfrac{3+d}{1-d}, \quad m_2 = \tfrac{2+\alpha}{\alpha} = 5.$$
Based on the definition of FLP, for any $\Cs h^{\beta}< x < \Cl h^{-\alpha} $ one can directly derive
\begin{equation}
\begin{aligned}
    \big\vert
        \hat\mu( x )
    \big\vert
    +
    \big\vert
        \hat\mu'( x )
    \big\vert
    \leq
    C \Big(
        1 + \vert x \vert^{\deg^+(\hat\mu)}
        +
        \vert x \vert^{\deg^-(\hat\mu) - 1}
    \Big).
\end{aligned}
\end{equation}
\begin{itemize}
    \item If $\deg^+(\hat\mu)<0$, then
    \begin{equation*}
        \big\vert
        \hat\mu( x )
        \big\vert
        +
        \big\vert
        \hat\mu'( x )
        \big\vert
        \leq
        C \Big(
        1 
        +
        \vert x \vert^{\deg^-(\hat\mu) - 1}
        \Big)
        \leq
        C \Big(
        1 
        +
        C h^{\beta (\deg^-(\hat\mu) - 1)}
        \Big)
        =
        C h^{-\frac12}.
    \end{equation*}

    \item If $\deg^-(\hat\mu)-1 < 0 < \deg^+(\hat\mu)$, then
    \begin{equation*}
        \big\vert
        \hat\mu( x )
        \big\vert
        +
        \big\vert
        \hat\mu'( x )
        \big\vert
        \leq
        C \Big(
        1 
        +
        C h^{-\alpha (\deg^+(\hat\mu))}
        +
        C h^{\beta (\deg^-(\hat\mu) - 1)}
        \Big)
        =
        C h^{-\frac12}.
    \end{equation*}

    \item If $\deg^-(\hat\mu)-1 > 0$, then
    \begin{equation*}
        \big\vert
        \hat\mu( x )
        \big\vert
        +
        \big\vert
        \hat\mu'( x )
        \big\vert
        \leq
        C \Big(
        1 
        +
        \vert x \vert^{\deg^+(\hat\mu)}
        \Big)
        \leq
        C \Big(
        1 
        +
        C h^{-\alpha (\deg^+(\hat\mu))}
        \Big)
        =
        C h^{-\frac12}.
    \end{equation*}
\end{itemize}
Hence one obtains
\begin{equation}\label{2024explicit-eq:CEV-hatmu_Ph_bound}
    \big\vert \hat\mu(x) \big\vert 
    +
    \big\vert \hat\mu'(x) \big\vert
    \leq C h^{-\frac12},
    \quad 
    \text{for all}
    \quad
    \Cs h^{\beta}< x < \Cl h^{-\alpha},
\end{equation}
which immediately verifies the first assertion in Assumption \ref{2024explicit-ass:hat_mu(Ph)_bound}.
Moreover, for any $\Cs h^{\beta}< x, y < \Cl h^{-\alpha} $ the mean value theorem infers that
\begin{equation}\label{2024explicit-eq:cor_hatmu_MVT}
\begin{aligned}
    \big\vert &  \hat\mu(\P(x)) -  \hat\mu(\P(y))  \big\vert^2\\
    & =
    \bigg\vert 
        \int_0^1 
            \hat\mu'(\theta \P(x) + (1 - \theta) \P(y)) 
        \d \theta
    \bigg\vert^2
    \cdot
    \big\vert  \P(x) - \P(y)  \big\vert^2
\end{aligned}
\end{equation}
Noticing that
$\Cs h^{\beta} \leq \theta \P(x) + (1 - \theta) \P(y) \leq \Cl h^{-\alpha}$,
one obtains from \eqref{2024explicit-eq:CEV-hatmu_Ph_bound} that
\begin{equation}\label{2024explicit-eq:cor_hatmu'_MVT_bound}
    \big\vert
        \hat\mu'(\theta \P(x) + (1 - \theta) \P(y))
    \big\vert^2
    \leq
    C h^{-1}.
\end{equation}
Substituting \eqref{2024explicit-eq:cor_hatmu'_MVT_bound} into \eqref{2024explicit-eq:cor_hatmu_MVT} gives
\begin{equation}
\label{2024explicit-eq:mu'_h(P_h_X-P_h_Y)_estimate}
    \big\vert  \hat\mu(\P(x)) - \hat\mu(\P(y))  \big\vert^2
    \leq
    C h^{-1}
    \big\vert  \P(x) - \P(y)  \big\vert^2.
\end{equation}
The second assertion in Assumption \ref{2024explicit-ass:hat_mu(Ph)_bound} follows by combining \eqref{2024explicit-eq:mu'_h(P_h_X-P_h_Y)_estimate} with the contractivity of $\P$ \eqref{2024explicit-eq:Ph_contractivity}.
}

The scheme \eqref{2024explicit-eq:scheme} now reads as
\begin{equation}\label{2024explicit-eq:scheme_for_CEV}
    Y_{n+1}
    =
    \P(Y_n)
    +
    \tfrac{d\sigma^2}{2(1-d)} h Y_{n+1}^{-1}
    +
    \hat\mu(\P(Y_n)) h 
    +
    \sigma \Delta W_n,
\end{equation}
with
\begin{equation}
    \P(x)
    =
    \min\{ \max\{ x , \Cs h^{\frac{1-d}{2}} \}, \Cl h^{-\frac12}\}.
\end{equation}
Simple rearrangements give
\begin{equation}
\begin{aligned}
    R_{n+1}
    & =
    X_{t_n} - \P (X_{t_n})
    +
    \int_{t_n}^{t_{n+1}}
        \mu(X_s) 
    \d s
    -
    c_{-1} h X_{t_{n+1}}^{-1}
    -
    \hat \mu(\P(X_{t_n})) h\\
    & =
    X_{t_n} - \P (X_{t_n})
    +
    \int_{t_n}^{t_{n+1}}
    \Big(    
        \mu(X_s) - \mu(X_{t_n})
    \Big)
    \d s
    +
    c_{-1} h
    \big( X_{t_n}^{-1} - X_{t_{n+1}}^{-1} \big)\\
    & \quad
    +
    h \big(
    \hat \mu(X_{t_n})
    -
    \hat \mu(\P(X_{t_n})) 
    \big).
\end{aligned}
\end{equation}
The It\^o formula implies that
\begin{equation}
\begin{aligned}
    R_{n+1}
    & =
    X_{t_n} - \P (X_{t_n})\\
    & \quad
    +
    \int_{t_n}^{t_{n+1}}
    \int_{t_n}^s
    \Big(
        \mu'(X_r)\mu(X_r) + \tfrac{\sigma^2}{2}\mu''(X_r)
    \Big)
    \d r \d s
    +
    \int_{t_n}^{t_{n+1}}
    \int_{t_n}^s
    \mu'(X_r) \sigma
    \d W_r \d s\\
    & \quad +
    c_{-1} h 
    \int_{t_n}^{t_{n+1}}
    \Big(
        -X_s^{-2}\mu(X_s) + X_s^{-3}\sigma^2
    \Big)
    \d s
    +
    c_{-1} h 
    \int_{t_n}^{t_{n+1}}
        -X_s^{-2}\sigma
    \d W_s\\
    & \quad
    +
    h \big(
    \hat \mu(X_{t_n})
    -
    \hat \mu(\P(X_{t_n})) 
    \big).
\end{aligned}
\end{equation}
Clearly one has
\begin{equation}
\begin{split}
    R_{n+1}^{(1)}
    & = 
    X_{t_n} - \P (X_{t_n})
    +
    \int_{t_n}^{t_{n+1}}
    \int_{t_n}^s
    \Big(
        \mu'(X_r)\mu(X_r) + \tfrac{\sigma^2}{2}\mu''(X_r)
    \Big)
    \d r \d s\\
    & \quad +
    c_{-1} h 
    \int_{t_n}^{t_{n+1}}
    \Big(
        -X_s^{-2}\mu(X_s) + X_s^{-3}\sigma^2
    \Big)
    \d s
    +
    h \big(
    \hat \mu(X_{t_n})
    -
    \hat \mu(\P(X_{t_n})) 
    \big),\\
    R_{n+1}^{(2)}
    & =
    \int_{t_n}^{t_{n+1}}
    \int_{t_n}^s
    \mu'(X_r) \sigma
    \d W_r \d s
    +
    c_{-1} h 
    \int_{t_n}^{t_{n+1}}
        -X_s^{-2}\sigma
    \d W_s.
\end{split}
\end{equation}
We focus on $R_{n+1}^{(1)}$ first.
Observing that $p^{*} = + \infty$, inequality \eqref{2024explicit-eq:correction_error} gives
\begin{equation}\label{2024explicit-eq:CEV_X-PhX_bound}
    \E \Big[
        \big\vert X_t - \P(X_t) \big\vert^{2q}
    \Big]
    \leq
    C h^{4q}
\end{equation}
for all $ q \geq 1$.
Inequality \eqref{2024explicit-eq:CEV_inverse_moment_bounds_lamperti} implies that
\begin{equation}
\begin{split}
    \E \Big[
        \big\vert 
        \mu'(X_r)\mu(X_r) + \tfrac{\sigma^2}{2}\mu''(X_r)
        \big\vert^{2q}
    \Big]
    & <
    +\infty,\\
    \E \Big[
        \big\vert 
        -X_s^{-2}\mu(X_s) + X_s^{-3}\sigma^2
        \big\vert^{2q}
    \Big]
    & <
    +\infty
\end{split}
\end{equation}
for all $q \geq 1$.
Thus the H\"older inequality ensures
\begin{equation}
\begin{split}
    \E\Bigg[
    \bigg\vert
        \int_{t_n}^{t_{n+1}}
        \int_{t_n}^s
        \Big(
            \mu'(X_r)\mu(X_r) + \tfrac{\sigma^2}{2}\mu''(X_r)
        \Big)
        \d r \d s
    \bigg\vert^{2q}
    \Bigg]
    & \leq
    C h^{4q},\\
    \E\Bigg[
    \bigg\vert
        c_{-1} h 
        \int_{t_n}^{t_{n+1}}
        \Big(
            -X_s^{-2}\mu(X_s) + X_s^{-3}\sigma^2
        \Big)
        \d s
    \bigg\vert^{2q}
    \Bigg]
    & \leq
    C h^{4q}
\end{split}
\end{equation}
for all $q \geq 1$.

Now it remains to estimate $\hat \mu(X_{t_n})
-
\hat \mu(\P(X_{t_n}))$.
The mean value theorem infers that
$$
\hat \mu(X_{t_n})
-
\hat \mu(\P(X_{t_n}))
=
\int_0^1
\hat\mu'\big( \theta X_{t_n} + (1-\theta) \P(X_{t_n}) \big)
\d \theta
\cdot
\big(
    X_{t_n}
    -
    \P(X_{t_n})
\big).
$$
By the H\"older inequality, one attains for $ q \geq 1$
\begin{equation}\label{2024explicit-eq:hat_mu_X-Phx_bound}
\begin{split}
    \E& \Big[
        \big\vert
        \hat \mu(X_{t_n})
        -
        \hat \mu(\P(X_{t_n}))
        \big\vert^{2q}
    \Big]\\
    & \leq
    \E\bigg[
        \int_0^1
        \Big\vert
            \hat\mu'\big( \theta X_{t_n} + (1-\theta) \P(X_{t_n}) \big)
        \Big\vert^{4q}
        \d \theta
    \bigg]^{\frac12}
    \cdot
    \E\Big[\big\vert
        X_{t_n}
        -
        \P(X_{t_n})
    \big\vert^{4q}\Big]^{\frac12}
\end{split}
\end{equation}
Noting that $\deg^+(\hat\mu)-1=0$ and $\deg^-(\hat\mu)-1 = - \frac{1}{1-d}$, by the property of FLP, it holds that
\begin{equation*}
    \vert \hat\mu'(x) \vert
    \leq
    C \Big(
        1 + \vert x \vert^{- \frac{1}{1-d}}
    \Big).
\end{equation*}
Thus one has
\begin{equation*}
\begin{split}
    \big\vert 
        \hat\mu'\big( \theta X_{t_n} + (1-\theta) \P(X_{t_n}) \big)
    \big\vert
    & \leq
    C \Big(
        1 + \vert X_{t_n} \vert^{- \frac{1}{1-d}}
        +
        \vert \P(X_{t_n}) \vert^{- \frac{1}{1-d}}
    \Big)\\
    & \leq
    C \Big(
        1 + \vert X_{t_n} \vert^{- \frac{1}{1-d}}
        +
        \vert h^{\beta} \vert^{- \frac{1}{1-d}}
    \Big)
\end{split}
\end{equation*}
Consequently, by \eqref{2024explicit-eq:CEV_moment_bounds_original} and \eqref{2024explicit-eq:CEV_alpha_beta} one concludes that
\begin{equation}\label{2024explicit-eq:hat_mu'_MVT_bound}
\begin{split}
    \E\bigg[
        \Big\vert
            \hat\mu'\big( \theta X_{t_n} + (1-\theta) \P(X_{t_n}) \big)
        \Big\vert^{4q}
    \bigg]
    \leq
    C h^{-2 q}.
\end{split}
\end{equation}
Substituting \eqref{2024explicit-eq:CEV_X-PhX_bound} and \eqref{2024explicit-eq:hat_mu'_MVT_bound} into \eqref{2024explicit-eq:hat_mu_X-Phx_bound} leads to
\begin{equation}
\begin{split}
    \E& \Big[
        \big\vert
        \hat \mu(X_{t_n})
        -
        \hat \mu(\P(X_{t_n}))
        \big\vert^{2q}
    \Big]
    \leq
    C h^{3 q}
\end{split}
\end{equation}
for all $q \geq 1$.
A combination of the above analysis leads to
\begin{equation}\label{2024explicit-eq:CEV_Rn1_estimate}
    \E\Big[\big\vert R_{n+1}^{(1)} \big\vert^{2q} \Big]
    \leq
    C h^{4q},
    \quad
    \forall q \geq 1.
\end{equation}
The estimation of $R_{n+1}^{(2)}$ closely resembles the approaches taken in the CIR process.
Equipped with \eqref{2024explicit-eq:CEV_inverse_moment_bounds_lamperti}, one obtains by the H\"older inequality, the Jensen inequality and the moment inequality that
\begin{equation}\label{2024explicit-eq:CEV_Rn2_estimate}
    \E\Big[\big\vert R_{n+1}^{(2)} \big\vert^{2q} \Big]
    \leq
    C h^{3q},
    \quad
    \forall q \geq 1.
\end{equation}
Bearing \eqref{2024explicit-eq:CEV_Rn1_estimate} and \eqref{2024explicit-eq:CEV_Rn2_estimate} in mind,
Theorem \ref{2024explicit-thm:main_convergence_result}  gives
\begin{equation}
\begin{split}
        \E\Big[ \sup_{n=0,...,M} \big\vert e_{n} \big\vert^{2p} \Big]
        & \leq
        2 C \text{e}^{2T}
        \Bigg(
            \tfrac{1}{h^{2p}}
            \sup_{n = 1,...,M}
            \E \Big[ \big\vert R_{n}^{(1)} \big\vert^{2p} \Big]
            +
            \tfrac{1}{h^{p}}
            \sup_{n = 1,...,M}
            \E \bigg[
            \big\vert  R_{n}^{(2)} \big\vert^{2p}
            \bigg] 
        \Bigg)\\
        & \leq
        C h^{2p},
\end{split}
\end{equation}
for any $p \geq 1$.

Finally, by Lemma \ref{2024explicit-lem:X_integrability} and Lemma \ref{2024explicit-lem:Y_integrability}, transforming back yields that for any $p \geq 1$,
\begin{align*}
    & \E \Big[
        \sup_{n = 0,...,M} 
        \big\vert  \mathbb{X}_{t_n} - \mathbb{Y}_n \big\vert ^{2 p}
    \Big]
    = (1-d)^{\frac{2p}{1-d}}
    \E \bigg[
        \sup_{n = 0,...,M} 
        \Big\vert  
            X_{t_n}^{\frac{1}{1-d}}
            - 
            Y_n^{\frac{1}{1-d}}
        \Big\vert ^{2 p}
    \bigg]\\
    & \leq
    C \cdot \E \bigg[
        \sup_{n = 0,...,M} 
        \Big\vert 
        X_{t_n}^{\frac{d}{1-d}} + Y_n^{\frac{d}{1-d}} 
        \Big\vert ^{2 p}
        \cdot
        \big\vert  X_{t_n} - Y_n \big\vert ^{2 p}
    \bigg]\\
    & \leq 
    C h^{2 p}.
\end{align*}
\begin{proposition}
\label{2024explicit-prop:CEVmodel}
For any $p \geq 1$, the proposed scheme is $2p$-strongly convergent with order $1$ for the CEV process.
\end{proposition}

\subsection{A\"it-Sahalia model}

The Ait-Sahalia model, introduced by Yacine Aït-Sahalia \cite{ait1996testing}, is a sophisticated nonlinear SDE that serves as a pivotal tool in financial mathematics for modeling the temporal evolution of interest rates. 
This model has gained considerable recognition for its ability to capture the complex dynamics of the spot rate and other financial variables, including volatility.
In \cite{szpruch2011strongly}, Higham et al. investigated a backward Euler method for a generalized version of the model, given by
\begin{align}
    \d \mathbb{X}_t
    =
    (\alpha_{-1} \mathbb{X}_t^{-1}
    -
    \alpha_0
    +
    \alpha_1 \mathbb{X}_t
    -
    \alpha_2 \mathbb{X}_t^r)
    \d t
    +
    \alpha_3 \mathbb{X}_t^\rho \, \d W_t,
\quad
t \in (0, T],
\quad
\mathbb{X}_0 = x_0 >0,
\end{align}
where $\alpha_{-1},\alpha_0,\alpha_1,\alpha_2, \alpha_3$ are positive constants and $r,\rho > 1 $.
In this paper we focus on the non-critical case, i.e., the case that $r+1 > 2 \rho$.
By applying the Lamperti-type transformation
$\mathbb{L}: (0,+\infty) \rightarrow (0,+\infty)$
of the form
$\mathbb{L}(x) := \tfrac{1}{\rho-1} x^{1 - \rho}$, 
one obtains the following SDE:
\begin{align}
        \d X_t
        =
        \mu(X_t) \d t
        +
        \sigma \d W_t,
\end{align}
where $\sigma=-\alpha_3$ and
\begin{equation}
\begin{aligned}
\label{2024explicit-eq:Ait-Sahalia_mu}
    \mu(x)
    & =
       - \alpha_{-1} 
        \big( (\rho-1) x \big)^{\frac{\rho + 1}{ \rho - 1 }}
        +
        \alpha_0 
        \big( (\rho-1) x \big)^{\frac{ \rho }{ \rho-1 }}
        -
        \alpha_1 
        \big( (\rho-1) x \big)
        +
        \tfrac{\sigma^2}{2} \rho   
        \big( (\rho-1) x \big)^{-1}\\
        & \quad
        +
        \alpha_2 
        \big( (\rho-1) x \big)^{\frac{ \rho - r }{\rho - 1}}.
\end{aligned}
\end{equation}
Recall Remark \ref{2024explicit-remark_c-1} and set 
$c_{-1}  = \tfrac{\sigma^2 \rho}{2 (\rho-1)}$.
Hence
\begin{equation}\label{2024explicit-eq:Ait_hat_mu_c-1_definition}
    \hat\mu(x)
    =
    - \alpha_{-1} 
    \big( (\rho-1) x \big)^{\frac{\rho + 1}{ \rho - 1 }}
    +
    \alpha_0 
    \big( (\rho-1) x \big)^{\frac{ \rho }{ \rho-1 }}
    -
    \alpha_1 
    \big( (\rho-1) x \big)
    +
    \alpha_2 
    \big( (\rho-1) x \big)^{\frac{ \rho - r }{\rho - 1}}.
\end{equation}
Note that $\mu$ and $\hat\mu$ are FLPs with 
\begin{align*}
    \deg^{-}(\mu) &=\deg^{-}(\hat\mu)= \tfrac{\rho - r}{\rho - 1},&
    \deg^{+}(\mu) &=\deg^{+}(\hat\mu)= \tfrac{\rho + 1}{\rho - 1},\\ 
    \coeff^{-}(\mu) &= \alpha_{2} (\rho-1)^{\frac{\rho - r}{\rho - 1}},&
    \coeff^{+}(\mu) &= -\alpha_{-1} (\rho-1)^{\frac{\rho+1}{\rho-1}}.
\end{align*}
Below we verify the assumptions imposed on the SDE model.
\begin{enumerate}
    \item From \cite{szpruch2011strongly}, it is known that Assumption \ref{2024explicit-ass:solution_existence_and_uniqueness} holds for $D = (0, +\infty)$.

    \item Assumption \ref{2024explicit-ass:mu_one_sided_Lipschitz} can be easily confirmed by observing that
    ${\deg}^{-}(\mu) < 0 < {\deg}^{+}(\mu)$
    and
    $\coeff^+(\mu) < 0 < \coeff^-(\mu)$,
    thanks to Lemma \ref{2024explicit-lem:FLP_properties}.

    \item It can also be found in \cite{szpruch2011strongly} that for any $q \in \mathbb{R}$ and $T>0$
        \begin{align}\label{2024explicit-eq:Ait_moment_bounds_original}
            \sup_{0 \leq t \leq T} \E \Big[ \vert \mathbb{X}_t \vert^q \Big]
            <
            \infty.
        \end{align}
    Consequently, one has for any $q \in \R$ and $T>0$
    \begin{align} \label{2024explicit-eq:Ait_inverse_moment_bounds_lamperti}
        \sup_{t \in [0,T]} 
        \mathbb{E} \Big[ \vert {X}_t \vert^{q} \Big]
        <
        +\infty.
    \end{align} 
    Thus Assumption \ref{2024explicit-ass:inverse_moment_boundedness_of_analytical_solution} is satisfied with $p^* = +\infty$.
\end{enumerate}
Set
\begin{equation}
    \P(x) = \min\{ \max\{ x , \Cs h^{\beta} \},  \Cl h^{- \alpha}\},
\end{equation}
where $\Cs,\Cl$ are arbitrary positive constants and
\begin{equation}\label{2024explicit-eq:Ait_alpha_beta}
    \beta = \tfrac{1}{2(1 -\deg^-(\hat\mu))} =\tfrac{\rho-1}{2(r-1)},
    \quad
    \alpha = \tfrac{1}{2\deg^+(\hat\mu)} = \tfrac{\rho-1}{2(\rho+1)}.
\end{equation}
Following the same lines as in the CEV setting, it can be verified that $\P$ fulfills Assumptions \ref{2024explicit-ass:Ph_properties} and \ref{2024explicit-ass:hat_mu(Ph)_bound} with
$$m_1 = \tfrac{2-\beta}{\beta} = \tfrac{4 r - \rho - 3}{\rho - 1},
\quad 
m_2 = \tfrac{2+\alpha}{\alpha} = \tfrac{5\rho + 3}{\rho -1}.$$
The scheme \eqref{2024explicit-eq:scheme} now read as
\begin{equation}\label{2024explicit-eq:scheme_for_Ait-Sahalia}
    Y_{n+1}
    =
    \P(Y_n)
    +
    \tfrac{\sigma^2 \rho}{2(\rho - 1)} h Y_{n+1}^{-1}
    +
    \hat\mu(\P(Y_n)) h 
    +
    \sigma \Delta W_n,
\end{equation}
with
\begin{equation}
    \P
    =
    \min \{ \max \{ x , \Cs h^{\frac{\rho-1}{2(r-1)}} \}, \Cl h^{-\frac{\rho-1}{2(\rho+1)}}\}.
\end{equation}
Simple rearrangements give
\begin{equation}
\begin{aligned}
    R_{n+1}
    & =
    X_{t_n} - \P (X_{t_n})
    +
    \int_{t_n}^{t_{n+1}}
        \mu(X_s) 
    \d s
    -
    c_{-1} h X_{t_{n+1}}^{-1}
    -
    \hat \mu(\P(X_{t_n})) h\\
    & =
    X_{t_n} - \P (X_{t_n})
    +
    \int_{t_n}^{t_{n+1}}
    \Big(    
        \mu(X_s) - \mu(X_{t_n})
    \Big)
    \d s
    +
    c_{-1} h
    \big( X_{t_n}^{-1} - X_{t_{n+1}}^{-1} \big)\\
    & \quad
    +
    h \big(
    \hat \mu(X_{t_n})
    -
    \hat \mu(\P(X_{t_n})) 
    \big).
\end{aligned}
\end{equation}
The It\^o formula implies that
\begin{equation}
\begin{aligned}
    R_{n+1}
    & =
    X_{t_n} - \P (X_{t_n})\\
    & \quad
    +
    \int_{t_n}^{t_{n+1}}
    \int_{t_n}^s
    \Big(
        \mu'(X_r)\mu(X_r) + \tfrac{\sigma^2}{2}\mu''(X_r)
    \Big)
    \d r \d s
    +
    \int_{t_n}^{t_{n+1}}
    \int_{t_n}^s
    \mu'(X_r) \sigma
    \d W_r \d s\\
    & \quad +
    c_{-1} h 
    \int_{t_n}^{t_{n+1}}
    \Big(
        -X_s^{-2}\mu(X_s) + X_s^{-3}\sigma^2
    \Big)
    \d s
    +
    c_{-1} h 
    \int_{t_n}^{t_{n+1}}
        -X_s^{-2}\sigma
    \d W_s\\
    & \quad
    +
    h \big(
    \hat \mu(X_{t_n})
    -
    \hat \mu(\P(X_{t_n})) 
    \big).
\end{aligned}
\end{equation}
Clearly one has
\begin{equation}
\begin{split}
    R_{n+1}^{(1)}
    & = 
    X_{t_n} - \P (X_{t_n})
    +
    \int_{t_n}^{t_{n+1}}
    \int_{t_n}^s
    \Big(
        \mu'(X_r)\mu(X_r) + \tfrac{\sigma^2}{2}\mu''(X_r)
    \Big)
    \d r \d s\\
    & \quad +
    c_{-1} h 
    \int_{t_n}^{t_{n+1}}
    \Big(
        -X_s^{-2}\mu(X_s) + X_s^{-3}\sigma^2
    \Big)
    \d s
    +
    h \big(
    \hat \mu(X_{t_n})
    -
    \hat \mu(\P(X_{t_n})) 
    \big),\\
    R_{n+1}^{(2)}
    & =
    \int_{t_n}^{t_{n+1}}
    \int_{t_n}^s
    \mu'(X_r) \sigma
    \d W_r \d s
    +
    c_{-1} h 
    \int_{t_n}^{t_{n+1}}
        -X_s^{-2}\sigma
    \d W_s.
\end{split}
\end{equation}
Observing that $p^{*} = + \infty$, and following the same approach as for the CEV model, it can be shown that for any $ q \geq 1$,
\begin{equation}\label{2024explicit-eq:Ait_Rn_estimate}
    \E\Big[\big\vert R_{n+1}^{(1)} \big\vert^{2q} \Big]
    \leq
    C h^{4q},
    \quad
    \E\Big[\big\vert R_{n+1}^{(2)} \big\vert^{2q} \Big]
    \leq
    C h^{3q}.
\end{equation}
Therefore,
Theorem \ref{2024explicit-thm:main_convergence_result}  gives
\begin{equation}
\begin{split}
        \E\Big[ \sup_{n=0,...,M} \big\vert e_{n} \big\vert^{2p} \Big]
        & \leq
        2 C \text{e}^{2T}
        \Bigg(
            \tfrac{1}{h^{2p}}
            \sup_{n = 1,...,M}
            \E \Big[ \big\vert R_{n}^{(1)} \big\vert^{2p} \Big]
            +
            \tfrac{1}{h^{p}}
            \sup_{n = 1,...,M}
            \E \bigg[
            \big\vert  R_{n}^{(2)} \big\vert^{2p}
            \bigg] 
        \Bigg)\\
        & \leq
        C h^{2p},
\end{split}
\end{equation}
for any $p \geq 1$.

Finally, by Lemma \ref{2024explicit-lem:X_integrability} and Lemma \ref{2024explicit-lem:Y_integrability}, transforming back yields that
\begin{equation}\label{2024explicit-eq:Ait_trans_back_before_analyzing_inverse_moment}
\begin{aligned}
    & \E \Big[
        \sup_{n = 0,...,M} 
        \big\vert  \mathbb{X}_{t_n} - \mathbb{Y}_n \big\vert ^{2 p}
    \Big]
    = (\rho-1)^{\frac{2p}{1-\rho}}
    \E \bigg[
        \sup_{n = 0,...,M} 
        \Big\vert  
            X_{t_n}^{\frac{1}{1-\rho}}
            - 
            Y_n^{\frac{1}{1-\rho}}
        \Big\vert ^{2 p}
    \bigg]\\
    & \leq
    C \cdot \E \bigg[
        \sup_{n = 0,...,M} 
        \Big\vert 
        X_{t_n}^{\frac{\rho}{1-\rho}} + Y_n^{\frac{\rho}{1-\rho}} 
        \Big\vert ^{2 p}
        \cdot
        \big\vert  X_{t_n} - Y_n \big\vert ^{2 p}
    \bigg].
\end{aligned}
\end{equation}
Lemma \ref{2024explicit-lem:inverse_moment_bounds_Y} together with \eqref{2024explicit-eq:Ait_Rn_estimate} infers that 
\begin{align*}
		\E \bigg[
	      \sup_{n = 0,...,M}
            \big\vert Y_n \big\vert^{-2q}
        \bigg]
        & \leq
        C \tfrac{1}{h^{2 q - 1}}
        \Bigg(
        \sup_{n = 0,...,M-1}
        \E \Big[ \big\vert R_{n+1}^{(1)} \big\vert^{2 q} \Big]
        +
        \sup_{n = 0,...,M-1}
        \E \Big[ \big\vert R_{n+1}^{(2)} \big\vert^{2 q} \Big]
        \Bigg)\\
        & \quad
        +
        C \E \bigg[
	      \sup_{t \in [0,T]}
            \big\vert X_t \big\vert^{-2q}
        \bigg]\\
        & \leq
        C h^q
        +
        C \E \bigg[
	      \sup_{t \in [0,T]}
            \big\vert X_t \big\vert^{-2q}
        \bigg].
	\end{align*}
Using Lemma \ref{2024explicit-thm:inverse_moment_bound_X} and \eqref{2024explicit-eq:Ait_inverse_moment_bounds_lamperti} one can further deduce that
\begin{equation}
\begin{split}
    \E \bigg[
	      \sup_{n = 0,...,M}
            \big\vert Y_n \big\vert^{-2q}
    \bigg]
    & \leq
    C h^q
    +
    C_q
        \sup_{t \in [0,T]} 
        \E \bigg[
			\vert X_t\vert ^{-(2 q + 2)}
		\bigg] 
   \leq
   C_q
\end{split}
\end{equation}
for any $q \geq 1$.
Consequently, it can be deduced from \eqref{2024explicit-eq:Ait_trans_back_before_analyzing_inverse_moment} that
\begin{equation*}
\begin{aligned}
    \E \Big[
        \sup_{n = 0,...,M} 
        \big\vert  \mathbb{X}_{t_n} - \mathbb{Y}_n \big\vert ^{2 p}
    \Big]
    \leq
    C h^{2p},
    \quad
    \forall p \geq 1.
\end{aligned}
\end{equation*}

\begin{proposition}
\label{2024explicit-prop:ASmodel}
For any $p \geq 1$, the proposed scheme is $2p$-strongly convergent with order $1$ for the A\"it-Sahalia model.
\end{proposition}

\section{Numerical Experiments}\label{2024explicit-sec:numerical_experiments}

This section presents numerical experiments to validate the theoretical findings.
We focus on evaluating the error decay rate.
Specifically, the approximation errors for the SDE models discussed in Section \ref{2024explicit-sec:examples} will be calculated in terms of 
$$
    e_M: = \bigg(  \E\Big[ \sup\limits_{n=0,1,...,M} \vert \mathbb{X}_{t_n} - \mathbb{Y}_n \vert^2 \Big] \bigg)^{\half}.
$$
The proposed explicit scheme \eqref{2024explicit-eq:scheme} and the LBEM will be both implemented for comparison.
Due to the absence of closed-form solutions and the exact expectations, two computational approximations are implemented:
\begin{itemize}
    \item Analytical solution approximation: A fine step-size LBEM scheme($h^{*} = 2^{-15}$) replaces the the analytical solution.
    \item Expectation approximation: The mathematical expectation is replaced by a Monte Carlo simulation with $M=10^4$ independent Brownian paths.
\end{itemize}
Under a fixed time $T=1$, we use step sizes $h = 2^{-i}$, $i= 5,6,7,8,9$ to investigate the numerical approximations and their convergence behavior of both the proposed scheme and LBEM.
The implicit equations arising in the LBEM implementation will be solved using Newton-Raphson iterations.
With these approaches, we obtain an approximation of $e_M$ of the following examples:
\begin{example}\label{2024explicit-eg:CIR}
   The CIR model with $\kappa = 0.35, \theta = 0.1, \hat\sigma = 0.1$,  $\mathbb{X}_0 = 0.1$ and $T=1$:
    \begin{equation}
	\d \mathbb{X}_t 
	= 
	0.35 (0.1 - \mathbb{X}_t) \, \d t 
	+ 
	0.1 \sqrt{\mathbb{X}_t} \, \d W_t , 
        \quad t \in [0,1], \qquad \mathbb{X}_0 = 0.1.
    \end{equation}
    Here the Feller index $\nu = \tfrac{2 \kappa \theta}{\sigma^2} = 7$.
    We choose $c_{-1}=2\big(\kappa \theta - \tfrac{\sigma^2}{4}\big)=0.065$ and $\P = I$ for the proposed scheme.
\end{example}

\begin{example}\label{2024explicit-eg:Heston}
The Heston-3/2 volatility with $a_1 = 0.8, a_2 = 0.1, a_3 = 0.5$, $\mathbb{X}_0 = \sin^2(0.9)$ and $T=1$:
\begin{align*}
    \d \mathbb{X}_t
    =
    0.8 \mathbb{X}_t \big( 0.1 - \mathbb{X}_t \big) \d t
    +
    0.5 \mathbb{X}_t ^{\frac{3}{2}} \, \d W_t.
\end{align*}
We choose $c_{-1} = 2(a_1 + \tfrac34 a_3^2) = 1.975$ and $\P = I$ for the proposed scheme.

\end{example}

\begin{example}\label{2024explicit-eg:CEV}
The CEV model with 
    $\kappa = 0.35, \theta = 0.1, \hat\sigma = 0.1$, $d=0.65$,  $\mathbb{X}_0 = 0.1$ and $T=1$:
    \begin{align}
    \d \mathbb{X}_t
    =
    0.35 (0.1 - \mathbb{X}_t) \d t
    +
    0.1 \mathbb{X}_t ^ {0.65}\, \d W_t.
    \end{align}
We choose $c_{-1}  = \tfrac{d \hat\sigma^2}{2(1 - d)} = \tfrac{13}{1400}$ and 
$$
    \P(x)
    =
    x \vee  h^{\frac{1-d}{2}} \wedge h^{-\frac12}
$$
for the proposed scheme.
\end{example}

\begin{example}\label{2024explicit-eg:Ait-Sahalia}
The A\"it Sahalia model with $ \alpha_{-1}=1.5, \alpha_0 = 2, \alpha_1=1, \alpha_2=2, \alpha_3=1, r=3, \rho = 1.5$, $\mathbb{X}_0 = 0.5$ and $T=1$:
\begin{align}
    \d \mathbb{X}_t
    =
    (1.5 \mathbb{X}_t^{-1}
    -
    2
    +
    \mathbb{X}_t
    -
    2\mathbb{X}_t^3)
    \d t
    +
    \mathbb{X}_t^{1.5} \, \d W_t,
\end{align}
We choose $c_{-1}  = \tfrac{\alpha_3^2 \rho}{2 (\rho-1)}=1.5$ and 
$$
\P(x)
    =
    x \vee h^{\frac{\rho-1}{2(r-1)}} \wedge 10 h^{-\frac{\rho-1}{2(\rho+1)}}
$$
for the proposed scheme.
\end{example}


\begin{figure}[htbp]
    \centering
    \begin{minipage}{0.48\linewidth}
        \centering
        \includegraphics[width = 1\linewidth]{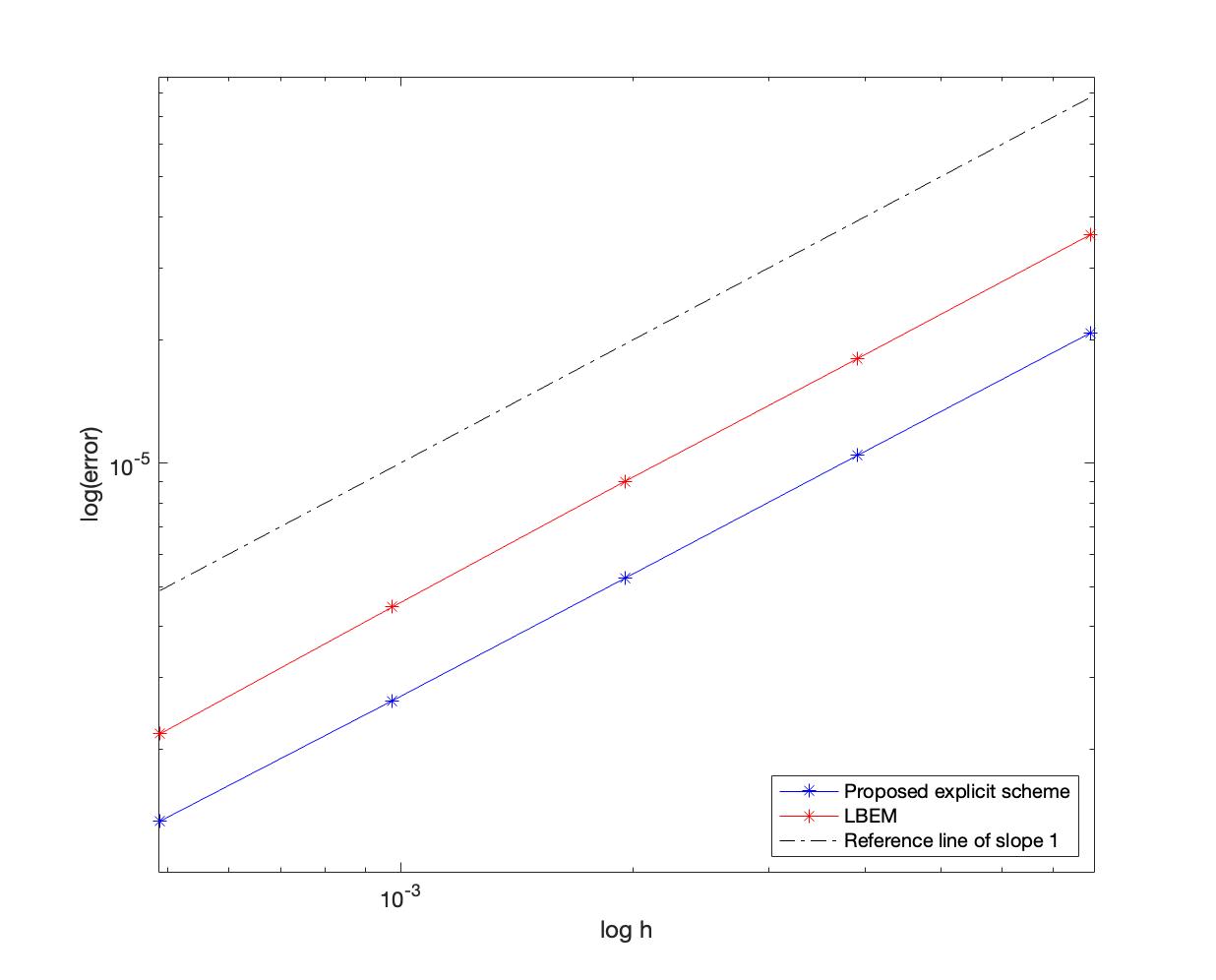}
        \caption{  CIR model  }
        \label{fig.CIR}
    \end{minipage}
    \begin{minipage}{0.5\linewidth}
        \centering
        \includegraphics[width = 1\linewidth]{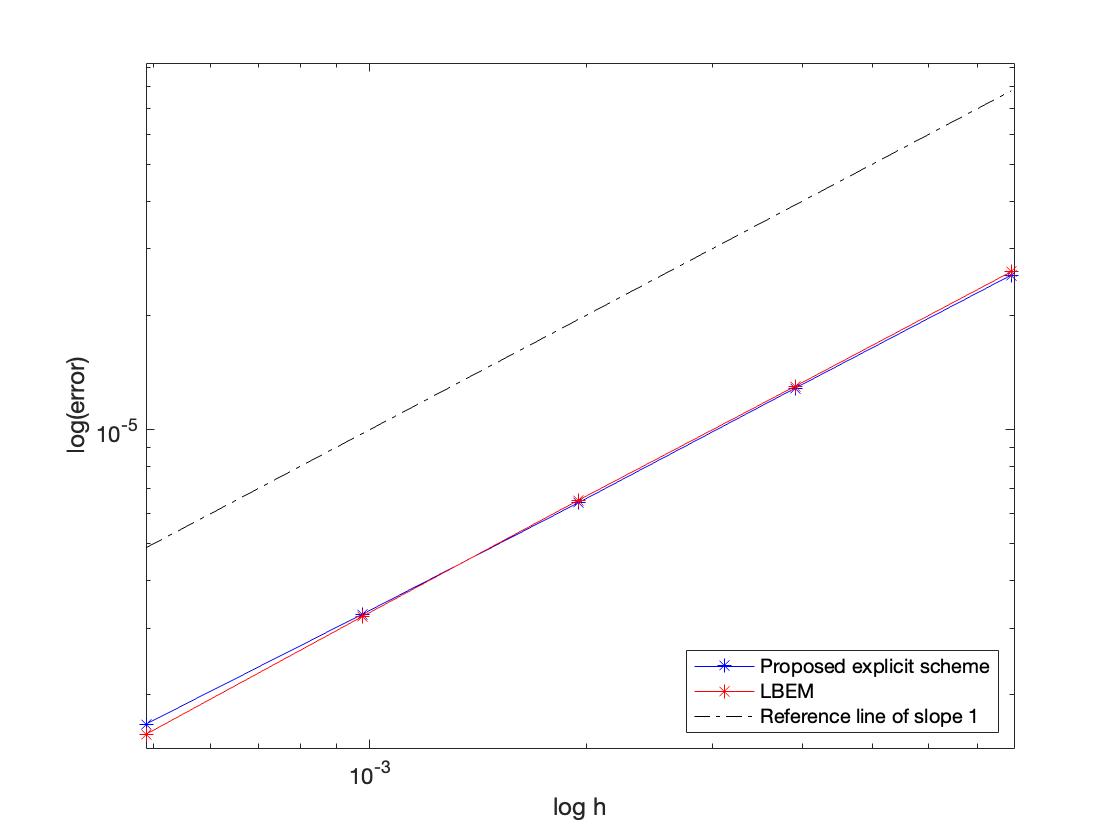}
        \caption{ Heston-3/2 model
        }
        \label{2024explicit-fig:Heston}
    \end{minipage}
\end{figure}

\begin{figure}[htbp]
    \centering
    \begin{minipage}{0.48\linewidth}
        \centering
        \includegraphics[width = 1\linewidth]{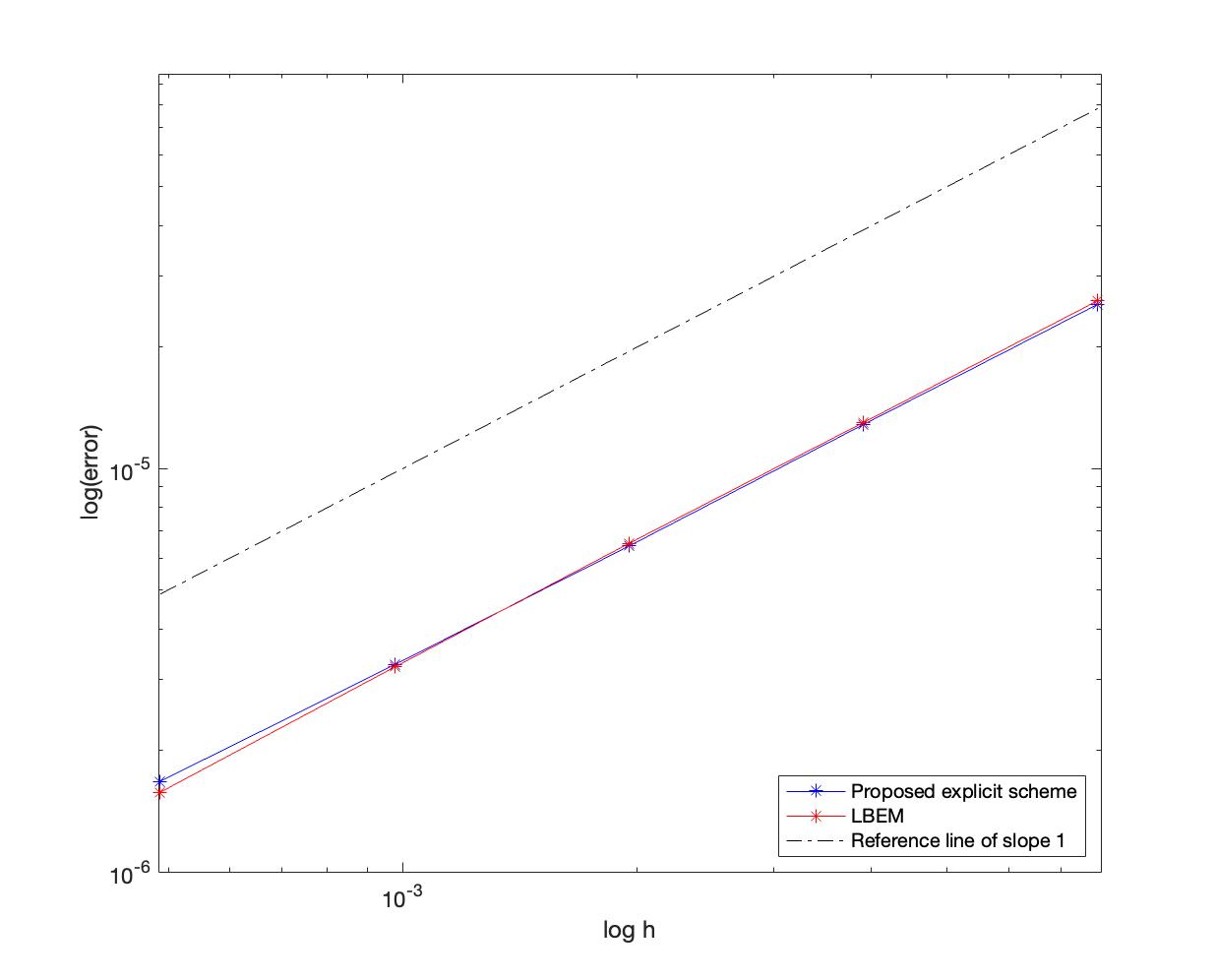}
        \caption{ CEV model }
        \label{fig.CEV}
    \end{minipage}
    \begin{minipage}{0.5\linewidth}
        \centering
        \includegraphics[width = 1\linewidth]{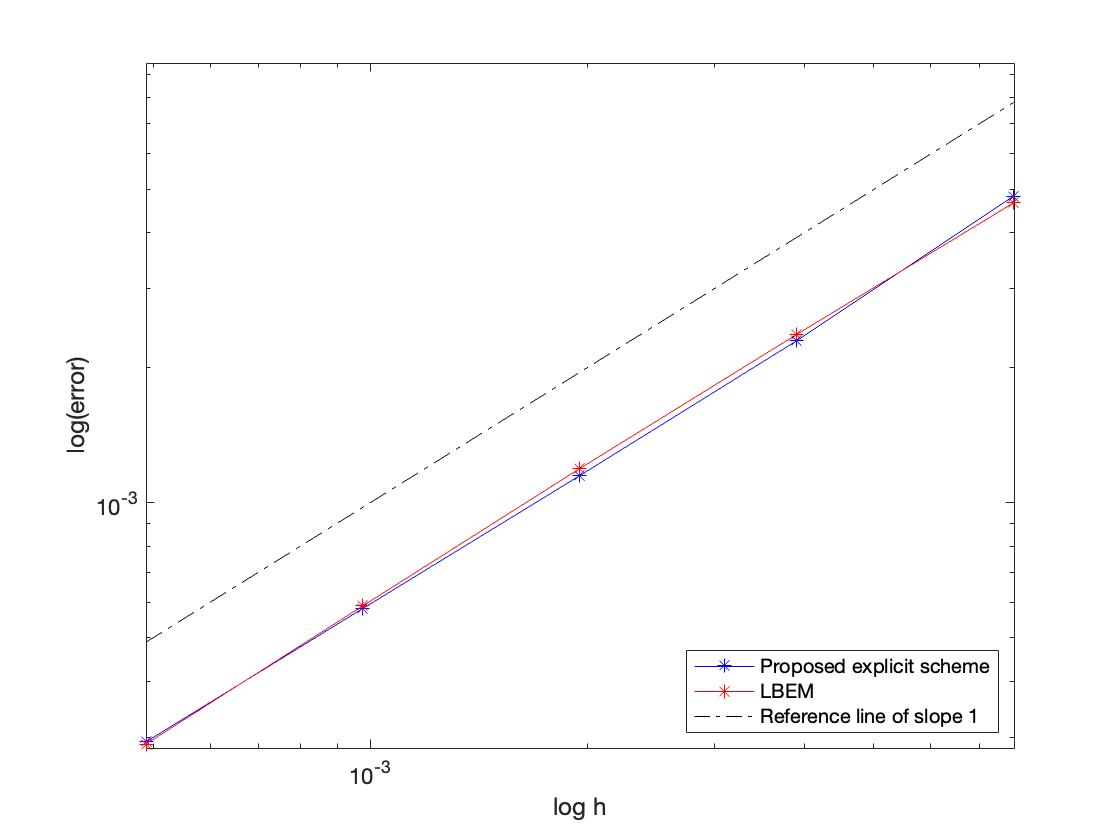}
        \caption{ A\"it-Sahalia model }
        \label{2024explicit-fig:Ait-Sahalia}
    \end{minipage}
\end{figure}

\begin{table}[htp]
    \centering
    \setlength{\tabcolsep}{4mm}
    \caption{ \bf Least-squares fit for the convergence rate q}
    \label{2024explicit-tab:convergence_rate}
    \begin{tabular}{l l l}
        \toprule[2pt]
            & Proposed explicit scheme & LBEM scheme\\
        \midrule
            Ex \ref{2024explicit-eg:CIR} (CIR) & 
            q=0.9909, resid=0.0070 & 
            q=1.0111, resid=0.0152 \\
            Ex \ref{2024explicit-eg:Heston} (Heston-3/2) & 
            q=1.0429, resid=0.0398 &
            q=1.0198, resid=0.0180\\
            Ex \ref{2024explicit-eg:CEV} (CEV) & 
            q=0.9838, resid=0.0167 & 
            q=1.0113, resid=0.0154\\
            Ex \ref{2024explicit-eg:Ait-Sahalia} (A\"it-Sahalia) &
            q=1.0073, resid=0.0378 & 
            q=1.0031, resid=0.0227\\
        \bottomrule[2pt]
\end{tabular}
\end{table}

\begin{table}[htp]
	\centering
	\setlength{\tabcolsep}{5mm}
    \caption{ \bf Time cost (seconds) over $10^4$ Brownian paths}
    \label{2024explicit-tab:time_costs}
	\begin{tabular}{l c c}
		\toprule[2pt]
		  & Proposed explicit scheme & LBEM\\
		\midrule 
		  Ex \ref{2024explicit-eg:CIR} (CIR) & 7.459 & 18.236 \\
            Ex \ref{2024explicit-eg:Heston} (Heston-3/2)  &  7.105  & 18.747\\
            Ex \ref{2024explicit-eg:CEV} (CEV) & 12.959 & 43.460 \\
            Ex \ref{2024explicit-eg:Ait-Sahalia} (A\"it-Sahalia) & 23.298 & 162.213 \\
		\bottomrule [2pt]
	\end{tabular}
	\vspace{2pt}
\end{table}

Figures \ref{fig.CIR}-\ref{2024explicit-fig:Ait-Sahalia} 
show a log-log plot of the step size $h$ versus the approximate error ${\tt err}_{M}^{N,h^*}$. 
The estimates of the proposed scheme are given by black lines, and the corresponding estimates for the LBEM scheme are given by red lines. 
The black dashed lines are reference ones with slope $1$.
From the figures, it can be observed that the proposed scheme achieves the same convergence rate of order $1$ as LBEM, while maintaining lower computational costs due to its explicit structure (See Table \ref{2024explicit-tab:time_costs}).
Notably, for the CIR model, the proposed scheme achieves superior error reduction under the same step sizes. 
We corroborate this also by performing a linear regression to estimate the convergence  rates.
By assuming that $ \log \left({\tt err}_{M}^{N,h^*} \right)
= \log C + q \log h$, the convergence rate $q$ and the least square residual can be obtained with a least-squares fitting, as presented in Table \ref{2024explicit-tab:convergence_rate}.
These results validate the expected convergence rate.

\section{Conclusion}\label{2024explicit-section:conclusion}
	In this manuscript, we proposed and analyzed an explicit time-stepping scheme for scalar SDEs defined in a domain. 
    Based on a Lamperti-type transformation and a taming procedure, this numerical scheme preserves the domain of the original equation and is strongly convergent with order $1$. 
    Our scheme is an explicit version of the Lamperti-backward Euler scheme from \cite{alfonsi2013strong,neuenkirch2014first} and has the same convergence order under the same conditions, but with lower computational costs. 
    Our explicit scheme is applicable to many SDEs from applications, and our theoretical findings are supported by numerical experiments.
    Future research directions in this area could be to further explore higher order explicit schemes for this class of equations.

\section{Appendix}
We present a discrete version of the Burkholder-Davis-Gundy inequality.
The following settings and theorem can be found in \cite[Theorem 1.1]{burkholder1972integral}.

Let $u=(u_1,u_2,...)$ be a martingale.
Denote
    \begin{align*}
        u^* = \sup_{1 \leq n < + \infty} \vert u_n \vert,
        \quad
        S(u) = \Bigg[ \sum_{k=1}^{+\infty} d_k^2 \Bigg]^{\frac12},
    \end{align*}
where $d = (d_1,d_2,...)$ is the difference sequence of $u$:
\begin{equation*}
    u_n = \sum_{k=1}^n d_k.
\end{equation*}

\begin{theorem}[Burkholder-Davis-Gundy]
\label{2024explicit-thm:discrete_BDG}
    Suppose that $\Phi$ is a convex function from $[0,+\infty)$ to $[0,+\infty)$ satisfying $\Phi(0) = 0$ and the growth condition
    \begin{equation*}
        \Phi(2 \lambda) \leq 2 C \Phi(\lambda),
        \quad
        \lambda > 0.
    \end{equation*}
    Set $\Phi(+\infty) = \lim_{\lambda \rightarrow +\infty} \Phi(\lambda)$.
    Then
    \begin{equation}
        c \E \Phi(S(u)) \leq \E \Phi(u^*) \leq C \E \Phi(S(u)).
    \end{equation}
\end{theorem}
We also quote from \cite[Lemma 10.2]{liu2025strong} a discrete form of the Gronwall lemma.
\begin{lemma}[Gronwall inequality]
\label{2024explicit-lem:Gronwall}
    Let $x_0=0$, $x_n \geq 0$ for $n=1,2, \ldots, M$ and  $\delta, \zeta, \eta \geq 0$.
    If
    $$
    x_{n+1} \leq \delta + \zeta \sum_{k=0}^n  x_k  + \eta \sqrt{x_M}, \quad n =0,1, \ldots, M-1,
    $$
    then one has
    $$
    x_{M} \leq 2(\delta +\eta^2) \exp(2 \zeta M).
    $$
\end{lemma}

\bibliographystyle{abbrv}
\bibliography{reference.bib}

 \end{document}